\documentclass[10pt]{amsart}

\usepackage{amsmath,amscd,amssymb}
\usepackage[mathscr]{eucal}
\usepackage[all]{xy}

\newcommand{\Wip}{\mathrm{A}_+^1}
\newcommand{\Id}{I}

\newcommand{\vanish}[1]{\relax}
\newcommand{\suchthat}{\,\,|\,\,}
\newcommand{\N}{\mathbb{N}}

\newcommand{\R}{\mathbb{R}}
\newcommand{\C}{\mathbb{C}}
\newcommand{\ud}{\mathrm{d}}
\newcommand{\ue}{\mathrm{e}}
\newcommand{\eM}{\mathrm{M}}
\newcommand{\vphi}{\varphi}
\newcommand{\Lap}{\mathcal{L}}
\DeclareMathOperator{\re}{Re}
\DeclareMathOperator{\im}{Im}
\newcommand{\konj}[1]{\overline{#1}}
\newcommand{\abs}[1]{\left\vert#1\right\vert}

\newcommand{\Ce}{\mathrm{C}}

\newcommand{\ohne}{\setminus}
\newcommand{\Dann}{\Longrightarrow}
\newcommand{\nach}{\circ}
\newcommand{\pfeil}{\longrightarrow}
\newcommand{\tpfeil}{\longmapsto}
\DeclareMathOperator{\dom}{dom}
\DeclareMathOperator{\ran}{ran}
\newcommand{\cls}[1]{\overline{#1}}

\DeclareMathOperator{\Lin}{\mathcal{L}}
\newcommand{\norm}[2][\relax]{%
   \ifx#1\relax \ensuremath{\left\Vert#2\right\Vert}
   \else \ensuremath{\left\Vert#2\right\Vert_{#1}}
   \fi}
\makeatletter
\newcommand{\sprod}[2]{\ensuremath{%
  \setbox0=\hbox{\ensuremath{#2}}
  \dimen@\ht0
  \advance\dimen@ by \dp0
  \left(\left.#1\rule[-\dp0]{0pt}{\dimen@}\,\right|#2\hspace{1pt}\right)}}
 \makeatother

\newcommand{\fourier}[1]{\widehat{#1}}

\newcounter{aufzi}
\newenvironment{aufzi}{\begin{list}{ {\upshape\alph{aufzi})}}{
        \usecounter{aufzi}
        \topsep1ex
        \parsep0cm
        \itemsep1ex
        \leftmargin0.8cm
        \labelwidth0.5cm
        \labelsep0.3cm
}}
{\end{list}}

\newcounter{aufzii}
\newenvironment{aufzii}{\begin{list}{\hfill {\upshape 
(\roman{aufzii})}}{
        \usecounter{aufzii}
        \topsep1ex
        \parsep0cm
        \itemsep1ex
        \leftmargin0.8cm
        \labelwidth0.5cm
        \labelsep0.3cm
         \itemindent0cm
}}
{\end{list}}

\newcounter{aufziii}

 \newtheorem{thm}{Theorem}[section]
 \newtheorem{cor}[thm]{Corollary}
 \newtheorem{lemma}[thm]{Lemma}
 \newtheorem{prop}[thm]{Proposition}

 \theoremstyle{definition}

 \theoremstyle{remark}
 \newtheorem{rem}[thm]{Remark}

\newtheorem{exa}[thm]{Example}
\newtheorem{remark}[thm]{Remark}

\numberwithin{equation}{section}

\numberwithin{equation}{section} \numberwithin{theorem}{section}

\newcommand{\Ces}{\mathrm{A}}
\newcommand{\rmo}{\mathrm{o}}
\newcommand{\rmO}{\mathrm{O}}

\begin{document}
\title[Bernstein functions and
rates in mean ergodic theorems]
 {Bernstein functions and  rates in mean ergodic theorems
for operator semigroups}

 \dedicatory{To Michael Lin on the occasion of his retirement}

\author{Alexander Gomilko}
\address{Faculty of Mathematics and Computer Science\\
Nicolas Copernicus University\\
Cho\-pin Str. 12/18\\
87-100 Toru\'n, Poland }

\email{gomilko@mat.uni.torun.pl}



\author{Markus Haase}

\address
{Delft Institute of Applied Mathematics\\
Delft University of Technology\\
P.O. Box 5031\\
2600 GA Delft, The Netherlands}

\email{m.h.a.haase@tudelft.nl}

\author{Yuri Tomilov}
\address{
Institute of Mathematics\\
Polish Academy of Sciences\\
\'Sniadeckich Str. 8\\
00-956 Warszawa, Poland,
and
Faculty of Mathematics and Computer Science\\
Nicolas Copernicus University\\
Chopin Str. 12/18\\
87-100 Torun, Poland\\}

\email{tomilov@mat.uni.torun.pl}




\date{\today}

\begin{abstract}
We present a functional calculus approach to the study of rates of decay in
mean ergodic theorems for bounded strongly continuous operator semigroups.
A central role is played by
 operators of the form $g(A)$, where $-A$ is the generator
of the semigroup and $g$ is a Bernstein function.
In addition, we  obtain some  new results
on Bernstein functions that are of independent interest.
\end{abstract}

\maketitle

\section{Introduction}

The famous mean ergodic theorem of von Neumann, Riesz, Kakutani, Lorch and
Eberlein states that for a power-bounded operator $T$ on
a reflexive Banach space $X$ the Ces\`aro averages
\[ \Ces_n(T) = \frac{1}{n} {\sum}^{n-1}_{j=0} T^j
\]
converge strongly as $n \to \infty$ to a bounded projection $P$
along $\cls{\ran}(\Id - T)$ onto the space
${\rm fix}(T) = \{ x\in X \suchthat Tx = x\}$
of fixed points. By
works of Butzer and Westphal \cite{BuWe71} and Browder
\cite{Bro58}, for $0 \not=x\in X$ the rate of convergence of $\Ces_n(T)x
\to Px$ cannot be ``better'' than $\rmO(1/n)$,
and this optimal rate happens if and only if $x\in \ran(\Id - T)$.
Moreover, by results of Dunford \cite{Dun43} and Lin
\cite{Lin74a}, in the case that $\ran(\Id - T)$ is not closed
there is no uniform rate of convergence working for all $x\in X$
simultaneously.  However,  one can
ask for conditions on individual vectors $x$ to guarantee a
certain rate, and Kachurovskii established in \cite{Ka96}
---
in the case of a unitary operator $T$ on a Hilbert space ---
connections between certain decay
rates of $\Ces_n(T)x$ and the spectral measure of $x$ with respect to $T$.

In 2001, Derriennic and Lin
--- motivated  by applications to central limit theorems for
Markov chains and the quest for rates in the strong law of large numbers ---
opened a new chapter by
addressing the problem of relating a prescribed decay rate for $\Ces_n(T)x$
with the convergence at $x$ of a certain power series in $T$.
The case of polynomial rates could be  settled already in \cite{DerLin01}
but some pertinent problems remained open, leading to a series
of subsequent papers \cite{AsLi07}, \cite{CoLi03}--\cite{Der06}.
In particular, it was asked in
\cite{AsLi07} whether the (weak) convergence of the power series
$\sum_{n=1}^\infty T^n x/n$ (the so-called {\em one-sided ergodic
Hilbert transform} of $x$) would imply that $\norm{\Ces_n(T)x} =
\rmO(1/ \log n).$ (A spectral characterization for this rate when $T$ is unitary
was given in \cite{Ga98} and \cite{AsLi07},
the case of normal $T$ was settled in \cite{CoLi09}.)
Taking into account the main results from
\cite{CoCuLi10} and \cite{HaTo10}, the question can be
reformulated as whether the Ces\`aro means $ \Ces_n(T)x$ decay
logarithmically if $x \in \dom \left(\log(I-T)\right)$. This
question was recently answered positively by the authors in
\cite{GoHaTo11}. Moreover, based on ideas from functional calculus
theory, a general method was given in \cite{GoHaTo11} to identify subspaces
where certain rates hold.

\medskip

In the present paper the analogous problems  for bounded
$C_0$-semigroups are discussed. Although the topic
is very natural, to the best of our knowledge
the only paper in this direction so far is \cite{KaRe10}. However,
as in the earlier work \cite{Ka96} it is
confined to unitary operators and spectral methods, so its thrust is quite different
from ours. Related results, but in the framework of real interpolation
spaces can be found in \cite{West1998}, \cite{Shaw}, and \cite{BuGe95}.

To set the stage, let $-A$ be the
generator of a bounded $C_0$-semigroup $T:= (T(s))_{s\ge 0}$ on a
complex Banach space $X$. We shall study the asymptotic behaviour of the
Ces\`aro averages
\begin{equation}\label{cesaromean}
\Ce_t(A)x:=\frac{1}{t} \int_0^t T(s)x\,\ud{s} \qquad (x \in X),
\end{equation}
as $t \to \infty$. It is easy to show that for $x, y \in X$
\[ \Ce_t(A)x \to 0 \quad \iff \quad x \in \cls{\ran}(A)
\]
and
\[ \Ce_t(A)x \to y \quad \Dann\quad y\in \ker(A).
\]
Therefore, $\ker(A) \oplus \cls{\ran}(A)$ is precisely the
subspace of $X$ on which the Ces\`aro averages converge strongly,
and the semigroup $(T(s))_{s \ge 0}$ is called {\em mean ergodic}
if  $X=\ker(A) \oplus \cls{\ran}(A)$ or, equivalently, $\Ce_t(A)$
converge strongly on $X$, see \cite[Theorem 18.7.3]{HilPhi}
or \cite[Section V.4]{EN}. A {\em mean ergodic theorem}
provides conditions for a semigroup to be mean ergodic; for instance,
a classical result states that every
bounded $C_0$-semigroup on a reflexive space is mean ergodic
\cite[Example V.4.7]{EN}.

\medskip

In this paper we shall  study  {\em rates of
convergence} for $\Ce_t(A)x$ as $t \to \infty$. Note that
\[ \ker(A) = {\rm fix}(T) := \{ x \in X \suchthat
T(s)x =x \,\, \forall\, s > 0\}.
\]
If $\Ce_t(A)x \to y$ then $y \in \ker(A)$ and hence
$\Ce_t(A)x - y = \Ce_t(A)(x - y)$. Thus,
in the study of rates of the convergence
of Ces\`aro averages one may confine oneself to the convergence
to zero on  $\cls{\ran}(A)$. By restricting
 $A$ to this $T$-invariant
subspace, there is no loss of generality in assuming that $X =
\cls{\ran}(A)$.

Let us recall now some known
facts.

\begin{prop}\label{int.p.rates}
Let $-A$ be the generator of a  bounded $C_0$-semigroup
$(T(s))_{s\geq 0}$ on a Banach space $X.$
Then the following statements hold.
\begin{aufzi}
\item If $\norm{\Ce_t(A)x} = \rmo(1/t)$ as $t \to \infty$, then $x= 0$.
\item If $x\in \ran(A)$ then $\norm{\Ce_t(A)x} = \rmO(1/t)$ as $t \to \infty$,
and the converse is true if $X$ is reflexive.
\item If there exists a positive function $\varphi:(0,\infty)\to (0,\infty)$
such that $\varphi(t) \searrow 0$ as $t\to\infty$, and
$\norm{\Ce_t(A)x} = \rmO(\varphi(t)), \, t \to \infty,$ for every $x\in X$, then $A$
is invertible.
\end{aufzi}
\end{prop}

Part a) is due, essentially, to Butzer and Westphal \cite{BuWe71}.
Actually, it follows easily from the formula
\[
A(I + A)^{-1} \frac{1}{t}\int_0^t s \, \Ce_s(A) x\,\ud{s} =
(I + A)^{-1}x-(I + A)^{-1}\Ce_t(A)x\qquad (t > 0)
\]
and the boundedness of the operator $A(I +A)^{-1}$. This result
tells us that we cannot have  better convergence rates than
$\rmO(1/t)$. The first assertion in part b) is straightforward,
and the second one is obtained in \cite[Theorem 2.8]{KrLi84}, see
also \cite{GoRaSho78}. For the proof of c)  one first concludes
that $\lim_{t \to \infty} \Ce_t(A) = 0$ in operator norm by  the
principle of uniform boundedness. Hence $T$ is a so-called {\em
uniformly ergodic} semigroup, and Lin has proved in \cite{Lin74b}
that for such operators $\ran(A)$ must be closed. At the same
time, under the assumption in c), we have $\ker(A)=\{0\}$ and
$\cls{\ran}(A) = X.$

 Actually, modifying Lin's
arguments one can sharpen this result to show that $A$-smoothness
of a  vector $x\in X$ has no influence on the asymptotics of
$\Ce_t(A)x$, see Theorem \ref{int.t.rates1} below.

The theory in the discrete case as developed in \cite{GoHaTo11}
hinges on the notion of an {\em admissible function}, and one of
the major difficulties was to find a continuous analogue for it.
As it turned out, the well-studied notion of a {\em Bernstein
function} provides such an analogue. However, the continuous
theory is by no means a straightforward translation of the
discrete theory, due to the fact that the generator $-A$ of a
$C_0$-semigroup is usually unbounded. More severely and very much
opposed to the discrete case, the operators $g(A)$, where $g$ is a
Bernstein function, are usually unbounded as well. Dealing with
this problem required a more sophisticated use of functional
calculus theory and some new results about Bernstein functions,
probably of independent interest.

As a result, we can cover polynomial and logarithmic rates.
Employing the notion of a {\em special Bernstein function} and using ideas
from \cite{CoCuLi11}, the lower estimates for rates could be
improved with respect  to \cite{GoHaTo11} to the extent that they now apply
under the sole condition that
$0$ is an accumulation point of the spectrum of the generator.
Moreover, in addition to what was considered in
\cite{GoHaTo11}, in this paper we {\em characterize} the functions that
arise as the rates of decay for Ces\`aro means in our setting
(Theorem \ref{rates.t.char}).

\medskip

Here is a synopsis of our main results:
Given a Bernstein function $g$  we establish
a uniform rate $r(t)$ of decay of $\Ce_t(A)x$ for $x \in \ran(g(A))$
(Theorem \ref{rates.t.rate-indiv}).
Then we characterize those rate functions $r$  that are associated
with a Bernstein function in this manner (Corollary
\ref{rates.c.rate-indivc} and Appendix \ref{AppendixB}).
Next, we show that the rate $r$ associated
with a Bernstein function $g$ can be read off from $g$ without
recurring on $g$'s representing measure, which is unknown in most cases
(Proposition \ref{rea.p.comp}). Then we prove that the (weak Abel) convergence
of a certain integral of the orbit $(T(s)x)_{s\ge 0}$ implies
that $x\in \ran(g(A))$, and hence that $\Ce_t(A)x$ has rate $r(t)$
(Theorem \ref{psr.c.final}). Finally we show that our results
are sharp under natural spectral assumptions,
see Theorem \ref{rea.t.main} and Remark \ref{remarkonrates}.
In Section \ref{s.exas} we illustrate our approach  with examples
for polynomial and logarithmic rates.

\subsection{Some Notations and Definitions}

For a closed linear operator $A$ on a complex Banach space $X$ we
denote by $\dom(A),$ $\ran(A)$, $\ker(A)$, and $\sigma(A)$ the
{\em domain}, the {\em range}, the {\em kernel}, and the {\em
spectrum} of $A$, respectively. The norm-closure of the range is
written as $\cls{\ran}(A)$. The space of bounded linear operators
on $X$ is denoted by $\Lin(X)$.  Let $\mathbb R_+$ stand for
$[0,\infty),$ and let $\eM(\R_+)$ denote the space of bounded
Radon measures on $\mathbb R_+$. We write $\C_+ :=\{ z\in \C \suchthat
\re z > 0\}$ for the open and $\cls{\C}_+ := \{ z\in \C\suchthat \re z \ge 0\}$
for the closed right halfplane.
For positive functions $r(t),
t\ge 0,$ and $s(t), t \ge 0,$ we write $r \sim s$ if there is $c
> 0$ such that $r(t)/c \le s(t) \le c r(t)$ for sufficiently large
$t \in \R_+$.

\section{Preliminaries}\label{prelim}

\subsection{Laplace Transforms}

A complex Radon measure $\mu$ on $\R_+$ is called {\em Laplace transformable}
if
\[ \int_{\mathbb R_+} e^{-st} \abs{\mu}(\ud{s}) <  \infty \quad
\text{for each $t > 0$}.
\]
Let us  write $\ue_z(s) := e^{-sz}$ for $z\in \C$ and $s \ge 0$. Then
$\mu$ is Laplace-transformable if $\ue_t \mu \in \eM(\R_+)$ for each $t> 0$.
The Laplace-transformable complex Radon measures form a Fr\'echet space.
The {\em Laplace transform} of a Laplace-transformable complex Radon
measure $\mu$ on $\R_+$ is
\[ (\Lap\mu)(z) = \fourier{\mu}(z) := \int_{\R_+} e^{-sz} \, \mu(\ud{s})
\qquad (\re z > 0).
\]
If $\mu$ is a bounded measure, then $\Lap\mu$ has an extension
to a  continuous function on $\cls{\C}_+$.
The space
\[ \Wip(\C_+) := \{ \Lap\mu \suchthat \mu \in \eM(\R_+)\}
\]
is a Banach algebra with respect to the pointwise multiplication
norm
\begin{equation}\label{mmm}
\norm{\Lap \mu}_{\Wip} := \norm{\mu}_{\eM(\R_+)} = \abs{\mu}(\R_+),
\end{equation}
and the Laplace transform
\[ \Lap : \eM(\R_+) \pfeil \Wip(\C_+)
\]
is an isometric isomorphism. Indeed, $\eM(\R_+)$ is a (unital) Banach algebra with multiplication given
by convolution  and with the norm defined by \eqref{mmm} (see \cite[p. 141-144]{HilPhi}), and the
Laplace transform is an injective algebra homomorphism
from $\eM(\R_+)$ to $\Wip(\C_+).$

More general, if $\mu, \nu$ are Laplace-transformable, then their convolution $\mu \ast \nu$
exists and is again Laplace-transformable. This follows from the identity
\[    \ue_t (\mu  \ast \nu) = (\ue_t\mu)\ast (\ue_t \nu)
\]
which is true for bounded measures, and can serve as a basis for
{\em defining} $\mu \ast \nu$ if $\mu$ or $\nu$ is not a bounded measure.
A simple computation then yields the identity
\[ \Lap(\mu\ast \nu) = (\Lap \mu) \, \cdot \, (\Lap \nu)
\]
for all Laplace-transformable Radon measures $\mu, \nu$ on $\R_+$.

\subsection{Functional Calculus}\label{s.hpfc}

Let $-A$ be the generator of a bounded $C_0$-semigroup
$(T(s))_{s\ge 0}$ on a Banach space $X$. Recall that at least heuristically `$T(s)=e^{-sA}$'.
Keeping this in mind, observe that
 the assignment
\[ g = \fourier{\mu} = \int_{\R_+} e^{-sz}\, \mu(\ud{s}) \quad
\mapsto \quad g(A) := \int_{\R_+} T(s)\, \mu(\ud{s})
\]
(with a strong integral in the definition of $g(A)$)
is a continuous algebra homomorphism of $\Wip(\C_+)$ into $\Lin(X)$
satisfying
\begin{equation}\label{hpfc.e.bound}
 \norm{g(A)} \le (\sup_{s\ge 0} \norm{T(s)} ) \, \norm{g}_{\Wip}
\qquad (g\in \Wip(\C_+)).
\end{equation}
This homomorphism is called the {\em Hille-Phillips} (HP) functional calculus
for $A$, see \cite[Chapter XV]{HilPhi}.
It  has a canonical extension towards a larger
function class, yielding unbounded operators in general.
This extension is constructed via the so-called {\em regularization method}
as follows: if $f: \C_+ \to \C$ is holomorphic such that
there exists a function $e\in \Wip(\C_+)$ with
$ef \in \Wip(\C_+)$ and the operator $e(A)$ is injective, then
\[ f(A) := e(A)^{-1} \, (ef)(A)
\]
with its natural domain $\dom (f(A)):=\{x \in X  \suchthat
(ef)(A)x \in \ran(e(A)) \}$. In this case $f$ is called
{\em regularizable}, and $e$ is called a {\em regularizer} for $f$. It is
easily shown that the definition of $f(A)$ does not depend on the
chosen  regularizer $e$ and that $f(A)$ is a closed
(but possibly unbounded) operator on $X$. Moreover, the set of all
regularizable functions $f$ is an algebra (depending on $A$).
(See e.g. \cite[p. 4-5]{Haa2006} and \cite[p. 246-249]{deLau95}.)
The assignment
\[ f \tpfeil f(A)
\]
from this algebra into the set of all closed operators on $X$ is
called the {\em extended Hille--Phillips calculus} for $A$. There
are natural rules governing this calculus, see for example
\cite[Chapter 1]{Haa2006}, the most important of which is the {\em
product rule}: {\em if $f$ is regularizable and $g\in \Wip(\C_+)$,
then}
\begin{equation}\label{hpfc.e.prod}
 g(A) f(A) \subseteq f(A) g(A) = (fg)(A),
\end{equation}
where we take the natural domain for a product
of operators, and inclusion means inclusion of graphs, i.e., extension.
In particular, it follows that $(fg)(A) \in \Lin(X)$
if and only if $\ran(g(A)) \subseteq \dom(f(A))$.

While the explicit description  of the domain of $f(A)$ could be rather nontrivial,
one can recover $f(A)$ from its restriction to $\dom (A^n), n \in \N,$ as the following lemma shows.
\begin{lemma}\label{hpfc.l.core}
Let $-A$ be the generator of a bounded $C_0$-semigroup on a Banach space $X$,
and let  $f(A)$ be defined in the extended
HP-calculus for $A$. Then for each $n \in \mathbb N$ the space
\[
{\mathcal I}:= \{ x\in \dom(A^n)\cap \dom(f(A)) \suchthat f(A)x \in \dom(A^n)\}
\]
is a core for $f(A)$, that is the closure of the restriction of $f(A)$ to ${\mathcal I}$ is $f(A)$ itself.
\end{lemma}

\begin{proof}
Note that $r_t(z) := (t+z)^{-1} \in \Wip(\C_+)$ for each $t > 0$.  By
\eqref{hpfc.e.prod}, if $x \in \dom (f(A))$ and
$f(A)x = y$ then $f(A)[t^n(t+A)^{-n}]x = t^n (t+A)^{-n} y$.
This shows that $t^n (t+A)^{-n}x \in \dom(f(A)) \cap \dom(A^n)$, and
since $\dom(A)$ is dense we have $t^n (t+A)^{-n}x \to x$ as $t \to \infty$,
and the same holds for $y$. As the operator $f(A)$ is closed this completes the proof.
\end{proof}

The following spectral inclusion theorem
is well-known \cite[Theorem 16.3.5]{HilPhi}. For the convenience of the reader
we provide a (particularly simple) proof.

\begin{thm}\label{hpfc.t.spin}
Let $g\in \Wip(\C_+)$ and let $-A$ be the generator of a bounded
$C_0$-semigroup $(T(s))_{s\ge 0}$ on a Banach space $X$. Then
\[  \{ g(\lambda) \suchthat \lambda \in \sigma(A)\}
= g(\sigma(A)) \subseteq  \sigma(g(A)).
\]
\end{thm}

\begin{proof}
Let $\lambda \in \sigma(A)$ and $g=\Lap \mu$ for some $\mu \in \eM(\R_+)$.
Note that  $\re \lambda \ge 0$, since the
semigroup is bounded. Hence we have
\[ g(\lambda) - g(z) = \lim_{N \to \infty} \int_0^N
(e^{-s\lambda} - e^{-sz})\, \mu(\ud{s})
\]
the convergence being in the norm of $\Wip(\C_+)$.  Furthermore,
\begin{align*}
  \int_0^N
(e^{-s\lambda} - e^{-sz})\, \mu(\ud{s}) & = - (\lambda - z)
\int_0^Ne^{-s\lambda}\int_0^s e^{t(\lambda - z)} \, \ud{t} \,
\mu(\ud{s})
\\ & = - (\lambda - z) \int_0^{N}
\left( e^{t\lambda} \int_{t}^N e^{-s\lambda} \, \mu(\ud{s})
\right) e^{-tz}\, \ud{t}
\\ & = (\lambda -z) h_N(z)
\end{align*}
for some  $h_N\in \Wip(\C_+)$. By \eqref{hpfc.e.prod}
\[ [(\lambda - z)h_N](A) = (\lambda - A) h_N(A),
\]
hence if $g(\lambda) - g(A)$ is
invertible then for big enough  $N  > 0$ the operator $(\lambda - A) h_N(A)$
is invertible as well.  Since $h_N(A)(\lambda - A) \subseteq (\lambda - A)
h_N(A)$ by \eqref{hpfc.e.prod} again, we conclude that $\lambda - A$ is
invertible.
\end{proof}

Suppose that $f = \Lap \mu$ for some Laplace-transformable
but not necessarily boun\-ded measure $\mu$. It is then natural to
examine the operator
\[ x \tpfeil \int_0^\infty T(s)x\, \mu(\ud{s})
\]
defined on the set of $x\in X$ where this integral exists
in whatever generalized sense. The next theorem roughly states that in
case of ``weak Abel summability'' of the integral, this operator
is in coherence with the extended HP-calculus.

\begin{thm}\label{hpfc.t.app-dom}
Let $\mu$ be a Laplace-transformable complex Radon measure on
$\R_+$, and let $f := \Lap\mu$. Suppose that $e \in \Wip(\C_+)$ is
such that $ef \in \Wip(\C_+)$ as well. Let $-A$ be the generator
of a bounded $C_0$-semigroup $(T(s))_{s\ge 0}$ on a Banach space $X$, and  let
$x,y\in X$ be such that
\begin{equation}\label{weakab}
 \lim_{\alpha \searrow 0} \int_0^\infty e^{-\alpha s} T(s)x\, \mu(\ud{s})
= y \qquad  \text{weakly}.
\end{equation}
Then $(ef)(A)x = e(A)y$.
\end{thm}

\begin{proof}
Since $\mu$ is Laplace transformable,  $\ue_\alpha \mu \in \eM(\R_+)$ for each $\alpha > 0$, and hence
\[ \int_0^\infty e^{-\alpha s} T(s) \, \mu(\ud{s}) = \Lap( \ue_\alpha \mu)(A).
\]
Let $\nu \in \eM(\R_+)$ such that $\Lap \nu = e$. Then, since
$ef \in \Wip(\C_+)$, $\nu \ast \mu \in \eM(\R_+)$. Moreover,
\[
\ue_\alpha\nu \to \nu\quad  \text{ and}\quad  \ue_\alpha(\nu \ast\mu) \to \nu \ast\mu
\qquad \text{as $\alpha \searrow 0$}
\]
in the norm of $\eM(\R_+)$. Consequently,
$\Lap(\ue_\alpha \nu)(A) \to e(A)$ in operator norm and hence
\[ \Lap(\ue_\alpha \nu)(A) \Lap( \ue_\alpha \mu)(A)x
\to e(A)y
\]
weakly. On the other hand,
\[
\Lap(\ue_\alpha \nu) \Lap(\ue_\alpha \mu)
=
\Lap( (\ue_\alpha \nu)\ast (\ue_\alpha \mu))
=
\Lap(\ue_{\alpha}(\nu \ast \mu)) \to \Lap(\nu \ast \mu) = ef
\]
in the norm of $\Wip(\C_+)$. Inserting $A$  concludes the proof.
\end{proof}

\begin{cor}\label{hpfc.c.app-dom}
Let $f = \Lap \mu$ and let $A, (T(s))_{s\ge 0}, x,y$
as in Theorem \ref{hpfc.t.app-dom}. Suppose that
$g$ is a (regularizable) holomorphic function on $\C_+$ so that $g(A)$ is defined
by the extended HP-calculus and $gf \in \Wip(\C_+)$. Then
$y \in \dom(g(A))$ and $(gf)(A)x = g(A)y$.
\end{cor}

\begin{proof}
Take any $e\in \Wip(\C_+)$ such that $eg \in \Wip(\C_+)$.
By Theorem \ref{hpfc.t.app-dom},
\[ e(A)(gf)(A)x = (egf)(A)x = (eg)(A)y.
\]
If $e$ is a regularizer for $g$, then $e(A)$ is injective, and
we can conclude that $(gf)(A)x = e(A)^{-1}(eg)(A)y$. Hence
$y \in \dom(g(A))$ and  $g(A)y = (gf)(A)x$.
\end{proof}

We note that the weak Abel summability  \eqref{weakab} is
weaker than improper weak summability
\begin{equation}\label{weakod}
\lim_{r \to\infty} \int_{0}^{r} T(s)x \, \mu(\ud{s})=y \qquad  \text{weakly}.
\end{equation}
Indeed, this follows by applying elements from the dual
space and employing the regularity of scalar Abel summability
\cite[p.181]{Wi41}.

Actually,  \eqref{weakab} is in general even strictly weaker than \eqref{weakod}.
As an example consider the case $f(z) = 1/z$, i.e., $\mu$
is ordinary Lebesgue measure. Then  \eqref{weakab} just means that
\begin{equation}\label{weakabinv}
  \lim_{\alpha \searrow 0} \int_0^\infty e^{-\alpha s} T(s)x\, \ud{s}
= \lim_{\alpha \searrow 0} (\alpha + A)^{-1} x  = y \qquad  \text{weakly}.
\end{equation}
Taking $g(z) = z$ in Corollary \ref{hpfc.c.app-dom} we obtain $x =
Ay$; conversely, it is easily seen that $x\in \ran(A)$ implies
\eqref{weakabinv}, cf.~\cite[Theorem~2.1]{KrLi84}. However, for
$x=Az \in \ran(A)$ one has
\[
\int_{0}^{r} T(s)x \, \mu(\ud{s}) = z - T(r) z,
\]
and hence \eqref{weakod} holds if and only if
$\lim_{r \to \infty} T(r)z = y-z$ weakly. Hence
every multiplication semigroup $T(s)=e^{ia s} \Id$ for $a \in \R\ohne \{0\}$
is an example for when \eqref{weakab} and \eqref{weakod} differ.
(See Remark \ref{hirsch} below for more about this topic.)

\subsection*{Bernstein Functions}

We now set up a functional-analytic background needed for our studies of rates.
As a general reference for most of material in the next subsections
we use the recent book \cite{SchilSonVon2010}.

The notions of completely monotone and Bernstein functions
are essential for our approach.
A function $f\in
\Ce^\infty(0,\infty)$ is called {\em completely
monotone} if
\[
f(t)\geq 0\quad \mbox{and} \quad  (-1)^n\frac{\ud^n f(t)}{\ud{t}^n}\geq 0
\qquad\text{for all $n \in \N$ and $ t > 0$.}
\]
By Bernstein's theorem \cite[Theorem 1.4]{SchilSonVon2010},
 a  function $f\in
\Ce^\infty(0,\infty)$ is completely monotone if and only if there
exists a (necessarily unique) Laplace-transformable positive Radon
measure $\mu$ on $\R_+$ such that $f(t) = (\Lap \mu)(t)$ for all
$t > 0$.

\smallskip

A function $g\in \Ce^\infty(0, \infty)$ is called a {\em Bernstein
function} if
\[
g(t)\geq 0\quad \mbox{and}\quad (-1)^n\frac{\ud^n g(t)}{\ud{t}^n}\leq 0
\qquad \text{for all $n \in \mathbb N$ and $ t > 0$}.
\]
By \cite[Theorem 3.2]{SchilSonVon2010},
a function $g$ is a Bernstein function  if and only if
there exist  constants $a, b\geq 0$ and a positive Radon measure
$\mu$ on $(0,\infty)$ satisfying
\begin{equation*}
\int_{0+}^\infty\frac{s}{1+s}\,\mu(\ud{s})<\infty \label{mu}
\end{equation*}
and such that
\begin{equation}\label{hpfc.e.bf}
g(z)=a+bz+\int_{0+}^\infty (1-e^{-sz})\mu(\ud{s})\qquad (z>0).
\end{equation}
The triple $(a, b, \mu)$ is uniquely determined by the
corresponding Bernstein function $g$. Note that from the
definition of $g$ it follows that $g$ extends analytically to
$\C_+$ and, moreover, $g \in \Ce(\cls{\C}_+)$ (see
\cite[Proposition 3.5]{SchilSonVon2010} and cf. Lemma
\ref{hpfc.l.bf-fc} below). Since such an analytic extension
is unique by standard complex function theory, there is no
harm in {\it identifying}
Bernstein functions with their extensions to $\C_+$,
and we shall henceforth do so.

Clearly, a Bernstein function $g \sim (a,b, \mu)$
is positive, increasing,  and satisfies
\[
a=g(0+)\quad \text{and} \quad b=\lim_{t \to \infty}\frac{g(t)}{t}.
\]
The Bernstein function $g$ is bounded if and only if $b=0$ and
$\mu(0,\infty)< \infty$  \cite[Corollary 3.7]{SchilSonVon2010}.
If $f$ is completely monotone and $g$ is a Bernstein function, then
$f \nach g$ is completely monotone \cite[Theorem 3.6]{SchilSonVon2010}.
In particular, if $0 \not= g$ is a Bernstein
function then
\[ \frac{1}{g} = \frac{1}{z} \nach g
\]
is completely monotone, hence by Bernstein's theorem there is a
positive Laplace transformable Radon measure $\nu$ with  $1/g =
\Lap\nu$. A completely monotone function $f$ is called a {\em
potential} if it is of the form $f =
1/g$ for some Bernstein function $g\not= 0$
\cite[Definition 5.17]{SchilSonVon2010}. (An analogous notion
was introduced and studied in \cite{ClPr90} where the representing
measure was called {\it completely positive}.) If $f$ is a
potential, then $f$ is decreasing  with
\[ \lim_{t \to \infty} f(t) = 0 \quad \text{if $g$ is unbounded},  \quad
\lim_{t\to \infty} tf(t) = \frac{1}{b}, \quad \lim_{t\searrow 0}
f(t)= \frac{1}{a},
\]
where we write $1/ \infty := 0$. In particular, we have
\[ f \in \Wip(\C_+) \quad \iff\quad f(0+) < \infty \quad \iff\quad
g(0+) > 0.
\]
It is not always easy to identify potentials. One way is by virtue
of {\em Hirsch's theorem} \cite{Hir74,Hir75}
saying that $f$ is a potential if for every $t>0$ the sequence
\[
\alpha_n:=(-1)^n\frac{f^{(n)}(t)}{n!}\qquad (n\ge 0)
\]
is log-convex, i.e., satisfies
\[
\alpha_n^2 \le \alpha_{n-1} \alpha_{n+1} \qquad(n\ge 1).
\]
Alternatively, a completely monotone function $f$ is a potential
if
\begin{equation}
 f(z) = c + \int_0^\infty e^{-sz} v(s)\, \ud{s} \qquad (\re z >
 0),
\end{equation}
where $c \ge 0$ and $v:(0,\infty) \to (0,\infty)$ is a decreasing
and log-convex function \cite[Theorem 10.23, Corollary
10.24]{SchilSonVon2010}. In \cite{Ka28}, fundamental for our paper
\cite{GoHaTo11}, the log-convexity of a sequence is shown to be
crucial for the study of  inverses of functions analytic on the
unit  disc. Kaluza's results in \cite{Ka28} parallel those of
Hirsch's, and in fact can be used to deduce Hirsch's theorem
mentioned above.

\medskip

We shall now show that Bernstein functions are always part of the
extended HP-functional calculus. For a related statement see
\cite[Theorem 1.6]{ClPr90}.

\begin{lemma}\label{hpfc.l.bf-fc}
Every Bernstein function $g$ can be written in the form
\[ g(z) = g_1(z) + z\, g_2(z), \qquad z >0,
\]
where $g_1, g_2 \in \Wip(\C_+)$.
\end{lemma}

\begin{proof}
Suppose that $g \sim(a,b,\mu)$ as in \eqref{hpfc.e.bf}. We  extend
naturally $\mu$ to $\mathbb R_+$ by setting $\mu(\{0\})=0$ and let
\[ \nu(\ud{s})  := \frac{s}{s+1} \mu(\ud{s}) \in \eM(\R_+).
\]
Then we write
\[ (1 - e^{-sz})\frac{s+1}{s} =
1 - e^{-sz} + z \frac{1}{s} \int_0^s e^{-rz}\, \ud{r}.
\]
Integrating this against $\nu$ with respect to the variable $s >
0$ and using Fubini's theorem, we obtain
\begin{align*}
 g(z) & - a - bz   = \int_{0+}^\infty (1 - e^{-sz})\, \mu(\ud{s})
= \int_{0+}^\infty (1 - e^{-sz}) \frac{s+1}{s}\, \nu(\ud{s})
\\ & = c - (\Lap\nu)(z) + z \int_0^\infty \frac{1}{s}\int_0^s
 e^{-rz}\, \ud{r}\, \nu(\ud{s})
\\ &
= c - (\Lap\nu)(z) + z \int_0^\infty \left[ \int_{r}^\infty
\frac{\nu(\ud{s})}{s}\right] \, e^{-rz}\, \ud{r}
\\ & = c - (\Lap\nu)(z) + z (\Lap\gamma)(z)
\end{align*}
with
\[ c = \nu(0, \infty) \ge 0 \quad\text{and}\quad
\gamma  (\ud{r}) = \left[\int_{r}^\infty \frac{\nu(\ud{s})}{s}\right] \,
\ud{r}.
\]
Note that $\gamma$ is a positive bounded measure of total mass
$\gamma(\R_+) = c$.
\end{proof}

As a consequence we obtain that every Bernstein function is
regularizable by any of the functions  $e_{\lambda}(z)=(\lambda +z)^{-1}$, $\re
\lambda
> 0$ (corresponding to the resolvents $(\lambda +A)^{-1}$).

\begin{cor}\label{hpfc.c.bf-fc}
Let $-A$ generate a bounded $C_0$-semigroup $(T(s))_{s\ge 0}$ on a
Banach space $X$, and let $g \sim (a,b,\mu)$ be a Bernstein
function. Then $g(A)$ is defined in the extended HP-functional
calculus. Moreover, $\dom(A) \subseteq \dom(g(A))$ and
\begin{equation}\label{phillips}
g(A)x = ax + bAx + \int_{0+}^\infty (I - T(s))x \, \mu(\ud{s})
\end{equation}
for each $x\in \dom(A),$ and $\dom(A)$ is a core for $g(A)$. If
$Ax = 0$, then $g(A)x = ax$, and if $a
> 0,$ then $\ran(g(A)) =X$ and $g(A)$ is invertible.
\end{cor}

\begin{proof}
That $g(A)$ is well defined follows immediately from Lemma
\ref{hpfc.l.bf-fc}. To prove the formula for $x\in \dom(A)$,
we insert $A$ in  the representation of $g$ derived in the proof of Lemma
\ref{hpfc.l.bf-fc} (using the definition of $g(A)$ via regularization) and obtain
\[ g(A)x -ax - bAx = cx - \int_0^\infty T(s)x\, \nu(\ud{s}) +
\int_0^\infty T(s)Ax\, \gamma (\ud{s}).
\]
Then we reverse the computation in the proof of Lemma \ref{hpfc.l.bf-fc}
using  the formula
\[ x - T(s)x =  \int_0^s T(r)Ax\, \ud{r} \qquad (x\in \dom(A))
\]
to arrive at \eqref{phillips}.
If $a > 0$ then $f = 1/g \in \Wip(\C_+)$,
hence $fg= 1$ and $f(A)g(A) \subseteq g(A)f(A) = \Id$. The
remaining statement follows from Lemma \ref{hpfc.l.core} with
$n=1.$
\end{proof}

\begin{remark}
The formula \eqref{phillips}
has been first obtained by Phillips in
\cite{Phil52}. In his approach,  the operator that we now denote by
$-g(A)$ is defined
as the generator of a certain semigroup {\em subordinate} in the sense
of Bochner to the semigroup whose generator is $-A$.
Curiously enough, Phillips does not use the notation ``$g(A)$'',
not even informally.  Balakrishnan \cite{Bala59} gave the
first definition of a proper unbounded functional calculus that
extends the Hille--Phillips calculus. His construction
covers subordinate semigroups, as he explains in
\cite[Section 5]{Bala59}, but  the formula
\eqref{phillips} is not explicitly treated.
Schilling \cite{Schil98}, probably
unaware of Balakrishnan's paper, gives an alternative description of
$g(A)$ for $g$ a {\em complete} Bernstein function, and uses it
to prove functional calculus properties such as the product rule
and  a composition rule, see also \cite[Chapter 12]{SchilSonVon2010}.
The general method of regularization that we use to extend the Hille--Phillips
calculus, was first described in full generality by  deLaubenfels in
\cite{deLau95} as ``Construction Two'', but is modelled
 on earlier work of McIntosh and Bade, see \cite[Sections 2.8 and 4.6]{Haa2006}.
Schilling cites deLaubenfels' paper
in the bibliography, but does not relate his functional calculus
to deLaubenfels' approach. This has been achieved now with our
Lemma \ref{hpfc.l.bf-fc} and  Corollary \ref{hpfc.c.bf-fc}.
\end{remark}

\subsection*{Complete and Special Bernstein Functions}

A Bernstein function is called a {\em complete Bernstein function}
if its representing measure has a completely monotone density
with respect to Lebesgue measure, see \cite[Definition 6.1]{SchilSonVon2010}.
A Bernstein function $g\not=0$ is called {\em special} if
$z/g(z)$ is again a Bernstein function.
By \cite[Proposition 7.1]{SchilSonVon2010},
if $g \not= 0$ is a complete Bernstein function, then so is $z/g(z)$.
Hence every non-zero complete Bernstein function is special.

By \cite[Theorem 10.3, Remark 10.4, (ii)]{SchilSonVon2010}, a Bernstein function $g \not=0$
is special if and only if  the associated potential $f = 1/g$ has
the Laplace transform representation
\begin{equation}\label{hpfc.e.sbf-rep}
 f(z) = c + \int_0^\infty e^{-sz} v(s)\, \ud{s} \qquad (\re z >
 0),
\end{equation}
where $c\ge 0$ and $v: (0,\infty)\to  (0,\infty)$ is decreasing.
Building on this, we can prove that $\abs{g(z)} \sim g(\abs{z})$,
a fact of independent interest that will be important in Section
\ref{s.rea} below. More precisely, we have the following result.

\begin{thm}\label{hpfc.t.sbf-est}
Let $g \not= 0$ be a special Bernstein function. Then
\begin{equation}\label{hpfc.e.sbf-est}
\frac{1}{3e} \,
\abs{g(z)} \le g(\abs{z}) \le 3e\, \abs{g(z)}
\end{equation}
for all $z$ with $\re z \ge 0$.
\end{thm}

\begin{proof}
By continuity of $g$ in the closed right half-plane, it suffices
to prove \eqref{hpfc.e.sbf-est} only for $\re z > 0$.  Let $f :=
1/g$, and $c,v$ as in \eqref{hpfc.e.sbf-rep}. We first establish
the right-hand inequality, which  is equivalent to
\begin{equation}\label{hpfc.e.sbf-est2}
 \abs{f(z)} \le 3e f(\abs{z})\qquad (\re z > 0).
\end{equation}
Employing  integration by parts for
(improper) Riemann--Stieltjes integrals
\cite[Theorem 3.3.1]{HilPhi}, we  compute with $t, \re z > 0$:
\begin{align*}
 f(z) & = c + \int_0^t e^{-zs}v(s)\, \ud{s} -
\frac{1}{z}\int_t^\infty v(s) \, \ud{e^{-zs}}
\\ & =  c+ \int_0^t e^{-zs}v(s)\, \ud{s} -  \frac{v(s)e^{-zs}}{z}\Big|_{s=t}^{s= \infty}
+ \frac{1}{z}\int_t^\infty e^{-sz}\ud{v(s)}
\\ & = c + \int_0^t e^{-zs}v(s)\, \ud{s} +  \frac{v(t)e^{-zt}}{z}
+  \frac{1}{z}\int_t^\infty e^{-sz}\ud{v(s)}.
\end{align*}
Hence, since $v$ is decreasing,
\begin{align*}
\abs{f(z)} & \le
c + \int_0^t v(s)\, \ud{s} +  \frac{v(t)e^{-t\re z}}{\abs{z}}
-  \frac{1}{\abs{z}}\int_t^\infty e^{-s\re z}\ud{v(s)}
\\ & \leq
c + \int_0^t v(s)\, \ud{s} +  \frac{e^{-t\re z}}{\abs{z}}
\left( v(t) - \int_t^\infty \ud{v(s)}\right)
\\  & \le
c + \int_0^t v(s)\, \ud{s} +  \frac{2 v(t)}{\abs{z}}
\le c + \left(1 + \frac{2}{t\abs{z}}\right) \int_0^t v(s)\, \ud{s}
\\ & \le
c + \left(1 + \frac{2}{t\abs{z}}\right) e^{t \abs{z}}
\int_0^t e^{-s \abs{z}} v(s)\, \ud{s}
\le \left(1 + \frac{2}{t\abs{z}}\right) e^{t \abs{z}} f(\abs{z}).
\end{align*}
Now we insert $t = 1/\abs{z}$ and arrive at \eqref{hpfc.e.sbf-est2}, concluding
the proof of the right-hand inequality in \eqref{hpfc.e.sbf-est}.

The left-hand inequality follows immediately be applying
the right-hand inequality to the special Bernstein function
$z/g(z)$.
\end{proof}

\subsection*{Stieltjes Functions}

A function $f : (0, \infty) \to \R_+$ is called a {\em Stieltjes}
function if it can be written as
\begin{equation}\label{hpfc.e.stieltjes}
f(z) = \frac{a}{z} + b + \int_{0+}^\infty \frac{\mu(\ud{s})}{z + s}
\qquad (z > 0),
\end{equation}
where $a, b \ge 0$ and $\mu$ is a positive Radon measure on $(0, \infty)$
satisfying
\[ \int_{0+}^{\infty} \frac{\mu(\ud{s})}{1 + s} < \infty.
\]
In this case, $\mu$ is called a {\em Stieltjes measure} and
\eqref{hpfc.e.stieltjes} is called the {\em Stieltjes representation} for $f$,
since such a representation is unique, see \cite[Chapter 2]{SchilSonVon2010}.

\begin{exa}\label{hpfc.exa.log}
We show that the function
\[ f(z) := \frac{\log z}{z-1}
\]
is a Stieltjes function. (Here and in the following, $\log z$ denotes
the principal branch of the complex logarithm.)
To this aim, we depart from the representation
\begin{equation}\label{hpfc.e.log-cbf}
 \log(1 + z) = \int_0^\infty (1 - e^{-sz}) e^{-s} \frac{\ud{s}}{s}
\end{equation}
valid for $\re z > -1$. (To see this, just take the derivative of the right-hand
side with respect to $z$.) By changing $z$ to $z -1$ we obtain
\[ {\log z} = \int_0^\infty (1 - e^{-s(z-1)}) e^{-s} \frac{\ud{s}}{s}
\]
valid for $\re z > 0$. Hence
\begin{align*}
\frac{\log z}{z-1} &= \int_0^\infty \frac{1 - e^{-s(z-1)}}{z-1} e^{-s}\,
 \frac{\ud{s}}{s} = \int_0^\infty \int_0^s e^{-t(z-1)} \, \ud{t}\,
 e^{-s}\, \frac{\ud{s}}{s} \\ &= \int_0^\infty \int_t^\infty
 \frac{e^{-(s-t)}}{s} \, \ud{s}\, e^{-tz}\, \ud{t} =
\int_0^\infty
 \left[\int_0^\infty \frac{e^{-s}}{s+t} \, \ud{s}\right]\, e^{-tz}\,
 \ud{t}.
\end{align*}
A change of variable $s \mapsto zs$ (with $z > 0$) and $t \mapsto st$
leads to
\begin{align*}
 \frac{\log z}{z-1} & =
\int_0^\infty
 \left[\int_0^\infty \frac{z e^{-sz}}{sz+t} \, \ud{s}\right]\, e^{-tz}\,
 \ud{t}
=
\int_0^\infty  \int_0^\infty
\frac{zs e^{-s(z+tz)}}{s(z+t)} \, \ud{s} \, \ud{t}
\\ &
=
\int_0^\infty  \frac{z}{z + zt} \frac{\ud{t}}{z + t}
=
\int_0^\infty  \frac{1}{z + t} \frac{\ud{t}}{1+ t},
\end{align*}
and this is a Stieltjes representation for $f$.
\end{exa}

Stieltjes functions are interesting in our context because
of the following result \cite[Proposition 7.1 and Theorem 7.3]{SchilSonVon2010}.

\begin{thm}\label{hpfc.t.stieltjes-cbf}
The function $f\not=0$ is a Stieltjes function if and only if $1/f$
is a complete Bernstein function if and only if $zf(z)$ is a complete Bernstein
function. In particular, every non-zero Stieltjes function is a potential.
\end{thm}

In particular, by Example  \ref{hpfc.exa.log} above,
the function
\[ f(z) = \frac{\log z}{z-1}
\]
is a potential, and $1/f$ is a complete Bernstein function. This
fact will be used in the following.

\medskip

The next result gives  a useful characterization of Stieltjes
functions \cite[Theorem 6.2 and Corollary 7.4]{SchilSonVon2010}.

\begin{thm}\label{hpfc.t.sti-char}
A non-zero function $f$ is a Stieltjes function if and only if $f$
admits an analytic extension to $\C \ohne (-\infty,0]$ such that
\begin{eqnarray*}
&& f(0+) := \lim_{t \searrow 0} f(t) \,\, \text{exists}, \quad  f(0+) \in (0, \infty], \\
&\text{and}&\\
&& \im z \cdot \im f(z) \le 0 \quad \text{for all $z \notin (-\infty, 0]$.}
\end{eqnarray*}
\end{thm}

Note that if $g(z) = \log(1 + z)$ then
\[ \frac{g(z) - g(1/z)}{z-1} = \frac{\log(1+z) - \log(1 + 1/z)}{z-1}
= \frac{\log z}{z-1},
\]
and, since
\[
\frac{e^{-s}}{s}=\int_{1}^{\infty}e^{-s\tau}\, \ud{\tau},
\]
$\log(1 + z)$ is a complete Bernstein function, by \eqref{hpfc.e.log-cbf}.
Hence the result of Example \ref{hpfc.exa.log}
is actually a consequence of the following theorem.

\begin{thm}\label{hpfc.t.cbf-sf}
Let $g\neq 0$ be a complete Bernstein function with $\lim_{t
\nearrow \infty} g(t)/t = 0$. Then
\[ f(z) := \frac{g(z) - g(1/z)}{z-1}
\]
is a Stieltjes function, hence a potential.
\end{thm}

\begin{proof}
By Theorem \ref{hpfc.t.sti-char} the function $t\mapsto g(t)/t$ is a
Stieltjes function. Since $\lim_{t \to \infty} g(t)/t = 0$,
we have
\[
g(z)=a +\int_{0+}^\infty \frac{z}{z+s}\, \rho(\ud{s})
\qquad   (z  > 0),
\]
with $a\geq 0$ and the representing measure $\rho$ satisfying
$\int_{0+}^\infty\,\frac{\rho(\ud{s})}{1+s}<\infty$.
Hence
\[
f(z) =\frac{1}{z-1}\int_{0+}^\infty
\left[\frac{z }{z+s} -\frac{1}{1+z s}
\right]\,\rho(\ud{s})
=
\int_{0+}^\infty \frac{(z+1)s\,\rho(\ud{s})}{(1+zs)(z+s)}
\qquad (z>0).
\]
So, the function $f:(0,\infty)\to (0,\infty)$ extends
analytically into $\C\setminus (-\infty,0]$. Moreover, for $s > 0$ and
$z\in \C$, $z\not=-s, -1/s$ we have
\begin{align*}
\abs{1 + zs}^2 & \abs{z + s}^2  \frac{(z+1)}{(z + s)(1 +zs)}
=
(z + 1) (1 + \konj{z}s) ( s + \konj{z})
\\ & =
s + (s^2 +1) \abs{z}^2 + \konj{z}( s^2 + 1 + s \abs{z}^2)
+ s \konj{z}^2 + sz.
\end{align*}
Taking imaginary parts, we obtain
\begin{align*}
\abs{1 + zs}^2 & \abs{z + s}^2  \im \frac{(z+1)}{(z + s)(1 +zs)}
\\ & = - (\im z) ( s^2 + 1 + s \abs{z}^2)  + (\im z)s  - 2 s(\im z)(\re z)\\
& = - (\im z)\left( s^2 + 1 + s \abs{z}^2 - s  + 2s\re z\right)
\\ & =
- (\im z)\left( (s -1)^2 + s (\abs{z}^2 + 2 \re z + 1)\right).
\end{align*}
Since $\abs{z}^2 + 2 \re z + 1 \ge (1 + \re z)^2 \ge 0$ for all $z\in \C$, we
see that
\[ \im z \cdot \im f(z)
= \frac{ - (\im z)^2 \left( (s -1)^2 + s (\abs{z}^2 + 2 \re z + 1)\right)%
}{ \abs{1 + zs}^2  \abs{z + s}^2} \le 0.
\]
whenever $z\in \C$ and $z \not= -s, -1/s$. Furthermore,
\begin{align*}
\frac{\ud}{\ud{t}}\left(
\frac{(t+1)}{(1+ts)(t+s)}\right)
& =\frac{ (1 + ts)(t +s) - (1 + t)(s^2 + 2ts + 1)
}{(1+ts)^2 (t+s)^2}
\\ &
= \frac{-(s-1)^2 - s(1 + t)^2
}{(1+ts)^2 (t+s)^2} \le 0
\end{align*}
for $t > 0$, and hence $\lim_{t \searrow 0} f(t) $ exists and
belongs to $(0, \infty]$. By  Theorem \ref{hpfc.t.sti-char} again,
$f$ is a Stieltjes function.
\end{proof}

\section{Estimating Rates in Terms of the Pre-Laplace Transform}\label{s.rates}

For $t > 0$  we define
\[ \Ce_t(z) := \frac{1}{t} \int_0^t e^{-sz}\, \ud{s}
= \frac{1 - e^{-tz}}{tz} \qquad (\re z \ge 0).
\]
Then $\Ce_t(z) = \Ce_1(tz) \in \Wip(\C_+)$ as well as $z \Ce_t(z) \in \Wip(\C_+)$,
with
\begin{equation}\label{rates.e.bound1}
\sup_{t > 0} \norm{\Ce_t}_{\Wip} + \sup_{t > 0} \norm{tz\Ce_t}_{\Wip} < \infty.
\end{equation}
For a Bernstein function $g \sim(a , b, \mu)$ we define
\begin{equation}\label{rates.e.r}
r(t) = r[g](t) := \frac{a}{2} + \frac{b}{t}  + \int_{0+}^\infty \min(s/t,1)\, \mu(\ud{s})
\qquad (t > 0).
\end{equation}
Note that by Fubini's theorem
\[ \int_{0+}^\infty \min(s/t,1)\, \mu(\ud{s}) =
\frac{1}{t} \int_{0}^t \mu(r,\infty)\, \ud{r} \qquad (t > 0).
\]
The following theorem is the reason why we are interested in the
function $r(t)$.

\begin{thm}\label{rates.t.rate}
Let $g \sim (a ,b ,\mu)$ be a Bernstein function and $r= r[g]$ as above.
Then  $\Ce_t \, g \in \Wip(\C_+)$ for each $t > 0$ and
\[ \norm{\Ce_t \, g}_{\Wip} = 2r(t) \qquad (t > 0).
\]
\end{thm}

\begin{proof}
Suppose first that $a = 0$. Applying Fubini's theorem twice we compute
\begin{align*}
t\Ce_t(z)g(z) & =
\int_0^t e^{-rz}\, \ud{r} \left(bz + \int_{0+}^\infty (1 - e^{-sz})\,
\mu(\ud{s})\right)
\\ & =
b(1 - e^{-tz}) + \int_0^t \int_{0+}^\infty \left(e^{-zr} -
e^{-z(r+s)}\right) \, \mu(\ud{s})\, \ud{r}
\\ & =
b(1 - e^{-tz}) + \int_{0+}^\infty \left(
\int_0^t e^{-zr}\, \ud{r} - \int_s^{s+t} e^{-zr} \,\ud{r} \right) \, \mu(\ud{s})
\\ & =
b(1 - e^{-tz}) + \int_{0+}^\infty \left(
\int_0^{\min(s,t)} e^{-zr}\, \ud{r}  - \int_{\max(s,t)}^{s+t} e^{-zr} \,\ud{r}
\right) \, \mu(\ud{s})
\\ & =
b(1 - e^{-tz}) +  \int_0^t \mu(r, \infty)\,  e^{-rz}\, \ud{r}
- \int_t^\infty \mu(r-t, r)\, e^{-zr}\,\ud{r}.
\end{align*}
This is the Laplace transform of the measure
$b\delta_0 - b\delta_t + \psi(r)\,\ud{r}$, where
\[
\psi(r)=\chi_{[0,t]}(r)\mu(r, \infty) -\chi_{(t, \infty)}(r)
\mu(r-t, r), \qquad (r \in \mathbb R_+).
\]
Hence we see that
\[ \norm{t\Ce_t \, g}_{\Wip} =
2b + \int_0^t \mu(r, \infty)\, \ud{r} +
\int_t^\infty \mu(r-t, r)\,\ud{r}.
\]
However, since $g(0+)=0$,  the two integrals here must be  equal, and hence
\[ \norm{t\Ce_t \, g}_{\Wip} = 2b + 2 \int_0^t \mu(r, \infty)\, \ud{r}
= 2t r(t)
\]
as claimed. If $a > 0$ then we have to add the term $a \int_0^t e^{-rz}\, \ud{r}$
in each line of the computation from above. This leads to the representation
\[ t\Ce_t(z) g(z) =
b(1 - e^{-tz}) +  \int_0^t (a + \mu(r, \infty))\,  e^{-rz}\, \ud{r}
- \int_t^\infty \mu(r-t, r)\, e^{-zr}\,\ud{r}
\]
and hence to the norm identity
\[ \norm{t\Ce_t \, g}_{\Wip} = at + 2b + 2 \int_0^t \mu(r,\infty)\, \ud{r}
= 2t r(t).
\]
\end{proof}

The next result lists some properties of the function $r[g]$, and
answers the question which functions $r$ on $(0, \infty)$ can arise
as $r = r[g]$.

\begin{thm}\label{rates.t.char}
For a  function $r : (0, \infty) \to (0, \infty)$ the following
assertions are equivalent.
\begin{aufzii}
\item  There exists a Bernstein function $0 \not= g\sim (a,b,\mu)$ such that
$r = r[g]$.
\item  The function $t\mapsto tr(t)$ is strictly positive, increasing and concave.
\end{aufzii}
Moreover, if {\rm (i)} or {\rm (ii)} is satisfied, then the
following assertions hold.
\begin{aufzi}
\item  $r$ is continuous and decreasing,  with $\lim_{t\to \infty} r(t) = a/2$.
\item  $\lim_{t\searrow 0} tr(t)= b$.
\item  The function $t \mapsto t r(t)$ is bounded if and only if  $a = 0$ and
there is $h\in \Wip(\C_+)$ such that
$g(z) = z h(z)$. In this case $h$ is completely monotone.
\end{aufzi}
\end{thm}

\begin{proof}
(i) implies (ii):\ Let $g \sim (a,b,\mu)$. Then
\[ r(t) = a/2 + b/t + \int_{0+}^\infty \min(s/t,1)\, \mu(\ud{s})
\]
is decreasing on $(0,\infty)$, continuous, and satisfies
$\lim_{t \to \infty} r(t) = a/2$ by the monotone convergence theorem.
Furthermore, the function
\[ f(t) := tr(t) = ta/2 + b + \int_0^t \mu(r,\infty)\, \ud{r}
\]
is increasing on $(0,\infty)$ and  satisfies $\lim_{t\searrow 0}f(t)  =
b$. Since the last summand is an integral of a decreasing
positive function, it is concave, and hence so is $f$. Note that
$f$ has one-sided derivatives
\[ D_+f(t) = \frac{a}{2}+ \mu(t,\infty)\quad \text{and}\quad
 D_-f(t) = \frac{a}{2} + \mu[t, \infty) \qquad (t > 0).
\]
(ii) implies (i):\ Define $f(t) := t r(t)$. Since $f$ is concave,
$f$ is absolutely continuous, has a right derivative $D_+f(t)$ at
each $t > 0$, and the function $D_+f$ is decreasing and right
continuous. Since $f$ is increasing, $D_+f(t) \ge 0$ for all
$t
> 0$. Define $a := 2 \lim_{t \to \infty} D_+f(t)$. Then the
function $m(t) := D_+f(t) - (a/2)$ is positive, right continuous
and decreases to $0$. By standard measure theory there exists a
positive Radon measure $\mu$ on $(0,\infty)$ satisfying $m(t) =
\mu(t,\infty)$ for all $t > 0$.

Since $f$ is absolutely continuous, we have
\[ f(t) - f(s) = \int_s^t  D_+f(r)\, \ud{r} =
\frac{a}{2}(t-s) + \int_s^t \mu(r, \infty)\, \ud{r}
\]
for all $0 < s < t < \infty$. Letting $s \to 0+$ here we obtain
\[ f(t) = b + \frac{a}{2}t + \int_0^t \mu(r,\infty)\, \ud{r}
\]
for any $t > 0,$ where $b:= \lim_{t \searrow 0} f(t).$ In
particular $\int_0^1 \mu(r, \infty)\, \ud{r} < \infty$, which
implies that
\[ \int_{0+}^\infty \min(1,s)\, \mu(\ud{s}) < \infty.
\]
Hence $r(t) = f(t)/t = r[g](t)$ for the Bernstein function $g \sim (a,b, \mu)$.

\smallskip

It remains to show c).  Clearly  $t r(t)$ is bounded on
$(0,\infty)$ if and only if $a= 0$ and
\begin{equation}\label{rates.e.bound2}
 \int_0^\infty \mu(r,\infty)\, \ud{r} = \int_{0+}^\infty s \, \mu(\ud{s}) < \infty.
\end{equation}
On the other hand, by Fubini's theorem,
\[ g(z) = bz + \int_{0+}^\infty (1 - e^{-rz})\, \mu(\ud{s})
= bz + z \int_0^\infty \mu(r, \infty) e^{-rz}\, \ud{r}
= zh(z),
\]
where $h$ is the Laplace transform of the positive measure
$\nu (\ud{r}) :=  b \delta_0 (\ud{r}) + \mu(r, \infty)\, \ud{r}$. The measure $\nu$ is finite
if and only if \eqref{rates.e.bound2} holds. The claimed equivalence now follows from
the injectivity of the Laplace transform.
\end{proof}

\begin{remark}
See Theorem \ref{t.cbf-rate} below for a related result on
rate functions associated with {\em complete} Bernstein functions.
\end{remark}

We now  employ the functional calculus. Let $-A$ be the generator of
a bounded semigroup $(T(s))_{s\ge 0}$, and let $M := \sup_{s\ge 0} \norm{T(s)}$.
If $g$ is a Bernstein function then, by Theorem \ref{rates.t.rate} and \eqref{hpfc.e.bound}, we have
\begin{equation}\label{rates.e.estimate}
 \norm{g(A)\Ce_t(A)} \le 2M r(t) \qquad (t > 0).
\end{equation}
Hence by \eqref{hpfc.e.prod} for $y = g(A)x \in \ran(g(A)),$
\[ \norm{\Ce_t(A)y}=\norm{g(A)\Ce_t(A)x}\le 2 M r(t)\norm{x} \qquad (t > 0).
\]
By the monotone convergence theorem, $\lim_{t\to \infty} 2r(t) =
g(0+)$. Hence, if $g(0+) > 0$  then nothing is
gained. This is no surprise since we have seen above that in this
case $\ran(g(A)) = X$, and there is no general convergence rate on
the whole space. However, in the case $g(0+)=0$ we obtain a
convergence rate to zero. Let us summarize our considerations.

\begin{thm}\label{rates.t.rate-indiv}
Let $g\sim (a,b,\mu)$ be a Bernstein function, and let $r = r[g]$
be as in \eqref{rates.e.r}. Let $-A$ be the generator of a $C_0$-bounded
semigroup $(T(s))_{s\ge 0}$ on a Banach space $X$ with $M :=
\sup_{s\ge 0} \norm{T(s)}$. Then the following statements hold.
\begin{aufzi}
\item  For each $y=g(A)x$
\begin{equation}\label{rates.e.rate-indiv}
 \norm{\Ce_t(A)y} \le 2 M r(t) \norm{x} \qquad (t > 0).
\end{equation}
\item  If
$tr(t) \to \infty$ as $t \to \infty$, $g(0+) = 0$ and
$(T(s))_{s\ge 0}$ is mean ergodic, then
\[ \norm{\Ce_t(A)y} = \rmo(r(t)) \qquad \text{as $t\to \infty$}
\]
whenever $y\in \ran(g(A))$.
\item If $tr(t)=
\rmO(1)$ as $t \to \infty,$ then $\ran(g(A)) \subseteq \ran(A)$
and
\[ \norm{\Ce_t(A)y} = \rmO({t}^{-1}) \qquad \text{as $t\to
\infty$}
\]
whenever $y\in \ran(g(A))$.
\end{aufzi}
\end{thm}

\begin{proof}
The estimate \eqref{rates.e.rate-indiv}
was obtained above.
To prove b), suppose that $a = g(0+) = 0$,
$t r(t) \to \infty$ as $t\to \infty$, and that the
semigroup $(T(s))_{s\ge 0}$ is mean ergodic.
By \eqref{rates.e.estimate}, the family of operators
\[ S_t := r(t)^{-1} g(A)\Ce_t(A),\qquad t > 0,
\]
is uniformly bounded. If $Ax = 0$ then $g(A) x = ax = 0$
(Corollary \ref{hpfc.c.bf-fc})
and hence $S_t x = 0$ for $t > 0$. On the other hand,
if  $x\in \dom(A)\subseteq \dom(g(A))$ then $y = Ax \in \ran(A)$ and using \eqref{hpfc.e.prod} we obtain
\[ S_t y = r(t)^{-1} g(A)\Ce_t(A)Ax = \frac{1}{tr(t)} [tA\Ce_t(A)]g(A)x
= \frac{1}{tr(t)}[tz\Ce_t(z)](A) g(A)x.
\]
Taking norms we infer that
\[ \norm{S_t y} \le \frac{M \norm{tz\Ce_t}_{\Wip} }{tr(t)} \norm{g(A)x}
\to 0   \qquad \text{as $t \to \infty$}
\]
by \eqref{rates.e.bound1}.  Since the semigroup is mean ergodic,
$\ker(A)\oplus \ran(A)$ is dense in $X$, and
hence $S_t \to 0$ strongly on $X$. It remains to note that for $y=g(A)x$
one has $S_t y = r(t)^{-1} \Ce_t (A) g(A) x$ as above.

For the proof of c)  suppose now  that $t r(t)$ stays bounded as
$t \to \infty$. Then by c) of Theorem \ref{rates.t.char} we have
$g(z) = zh(z)$ for some $h\in \Wip(\C_+)$, and hence $g(A) = A
h(A)$ by the functional calculus (see \eqref{hpfc.e.prod}). This
implies that $\ran(g(A)) \subseteq \ran(A),$ and then $y=Ax$ for
some $x \in \dom (A).$ Thus $\Ce_t(A)y=t^{-1}(x-T(t)x), t>0,$ so
that $\norm{\Ce_t(A)y}=\rmO(t^{-1})$ as $t \to \infty$ (cf.
Proposition \ref{int.p.rates}, b)).
\end{proof}

As a direct consequence of
Theorems \ref{rates.t.char}
and \ref{rates.t.rate-indiv} we state the following corollary.

\begin{cor}\label{rates.c.rate-indivc}
Let $r : (0, \infty) \to (0, \infty)$ be such that $t\mapsto
tr(t)$ is strictly positive, increasing and concave. Then
there is a Bernstein function $g$ such that $r= r[g]$, and hence
the conclusion of Theorem \ref{rates.t.rate-indiv} holds.
\end{cor}

Thus, in particular, any function $r$ subject to the condition (ii) of Theorem \ref{rates.t.rate-indiv}
 can be realized
as a rate of decay of $\Ce_t(A)$ restricted to the range of $g(A).$

\section{Estimating Rates in Terms of Laplace Transforms}\label{s.rea}

So far, the rate $r$ is given in terms of the measure $\mu$ from
the representation \eqref{hpfc.e.bf} of the Bernstein function $g$.
However, in
situations of interest we often only know the measure $\nu$
corresponding to a potential $f = 1/g$, and so it
seems desirable to be able to read off $r$ (or at least its
asymptotic behaviour) from the {\em values} of $f$ at certain
points.

To achieve this, we begin with  some elementary considerations involving the
simple inequalities
\begin{align*}
x e^{-x} & \le \min(2,x,1-e^{-x})  \quad\,\,(x > 0),\\
\abs{1 - e^{-zs}} & \le \min(2, s \abs{z})  \qquad\qquad
(s, \re z \ge 0).
\end{align*}

\begin{lemma}\label{rea.l.bf-est}
Let $g \sim(a,b,\mu)$ be  a Bernstein function, and $r := r[g]$. Then
\[ \left[(t \re z) e^{-t \re z}\right]\,\, r(t) \le \re g(z)\le \abs{g(z)} \le
\max(2,t \abs{z})\,\, r(t)
\]
for all $t > 0$ and $z\in \C$ with $\re z > 0$.
\end{lemma}

\begin{proof}
We have
\begin{align*}
\abs{g(z)} & \le a + b \abs{z} + \int_{0+}^\infty \abs{1- e^{-sz}}\, \mu(\ud{s})
\\ & \le
\frac{a}{2} 2  + b \abs{z} + \int_{0+}^t s \abs{z}\, \mu(\ud{s})
+ \int_{t+}^\infty 2 \, \mu(\ud{s})
\le \max(2, t \abs{z})\,\, r(t)
\end{align*}
for the upper estimate. For the lower estimate we write $x := t \re z > 0$. Then
\begin{align*}
\re g(z) & \ge  a + b \re z + \int_{0+}^\infty (1 - \re e^{-zs})\, \mu(\ud{s})
\\ &
\ge a + b \re z + \int_{0+}^\infty (1 - e^{-s \re z})\, \mu(\ud{s})
\\ &
\ge a + b \re z +  \int_{0+}^t (s \re z) e^{-s \re z}\, \mu(\ud{s})
+ (1 - e^{-t \re z}) \int_{t+}^\infty \, \mu(\ud{s})
\\ &
\ge 2 \frac{a}{2} + \frac{b}{t} x + xe^{-x} \int_{0+}^t s/t \, \mu(\ud{s})
+ (1 - e^{-x}) \mu(t, \infty)
\ge x e^{-x} r(t).
\end{align*}
\end{proof}

As a consequence we find that one can read off the (asymptotics of
the) rate $r(t)$ from values $g(z_t)$ if the set $(z_t)_{t > 0}$
is carefully chosen.

\begin{prop}\label{rea.p.comp}
Given $0 < \alpha \le \beta < \infty$ there are positive numbers
$c_0 = c_0(\alpha, \beta), c_1= c_1(\alpha, \beta)$ such
that the following holds.
Suppose that  $g$ is a Bernstein function with associated rate
function $r= r[g]$. Then
\[  c_0 \,r(t)\,\, \le\, \abs{g(z)}\, \le\,\, c_1\, r(t)
\]
whenever  $t> 0$ and  $\re z > 0$ are such that
$\alpha \le t \re z \le t \abs{z} \le \beta$.
\end{prop}

\begin{proof}
By Lemma \ref{rea.l.bf-est} we can choose
$ c_0(\alpha, \beta) = \inf_{\alpha\le x \le \beta} \left(xe^{-x}\right)
\quad \text{and}
\quad c_1(\alpha, \beta) = \max(2, \beta)$.
\end{proof}

For {\em special} Bernstein functions $g$ we can employ Theorem
\ref{hpfc.t.sbf-est}
and obtain a better result.

\begin{prop}\label{rea.p.sbf}
Given $0 < \alpha \le \beta < \infty$ there are positive numbers $c_0=c_0(\alpha, \beta), c_1=c_1
(\alpha, \beta)$ such that the following holds. If $g$ is a special Bernstein
function with associated rate function $r := r[g]$, then
\[
c_0\,  r(t) \,\, \le \,  \abs{g(z)}\,  \le \,\, c_1\, r(t)
\]
whenever $t > 0$ and $\re z \ge 0$ are such that $t \abs{z} \in
[\alpha, \beta]$. In particular, $c_0=1/(3e^2)$ if
$\alpha=\beta=1.$
\end{prop}

\begin{proof}
Combining Theorem \ref{hpfc.t.sbf-est} with Lemma \ref{rea.l.bf-est} we obtain
\[  \left(\frac{t\abs{z}e^{-t\abs{z}}}{3e}\right) r(t) \le \abs{g(z)}
\le \max(2, t\abs{z})\,  r(t) \qquad (t > 0, \re z \ge 0).
\]
Then choose $c_0 = (3e)^{-1} \inf_{\alpha \le x \le \beta} \left (x e^{-x}\right)$ and
$c_1 = \max(2, \beta)$.
\end{proof}

For a potential function  $f$, we can take $e = 1/f$ in Theorem
\ref{hpfc.t.app-dom} and combining it with Theorem
\ref{rates.t.rate-indiv} and Proposition \ref{rea.p.comp} we
obtain the following statement.

\begin{thm}\label{psr.c.final}
Let  $\mu$ be a positive Radon measure on $\mathbb R_+$ and let
$f=\Lap\mu$ be a potential function (e.g., a Stieltjes
function) with $f(0+) = \infty$.  Let $-A$ be the generator of a bounded
$C_0$-semigroup $(T(s))_{s\ge 0}$ on a Banach space $X$,  and let $x, y\in
X$ be such that
\[ \lim_{\alpha \searrow 0} \int_0^\infty e^{-\alpha s} T(s)x\, \mu(\ud{s})
= y \qquad  \text{weakly}.
\]
Then
\begin{aufzi}
\item  $x = (1/f)(A)y$  \quad and \quad
$\displaystyle \norm{\Ce_t(A)x} ={\rm O}\left(\frac{1}{f(1/t)}\right)
\qquad\text{as $t \to \infty$}$.
\item
$\displaystyle \norm{\Ce_t(A)x} = \rmo\left(\frac{1}{f(1/t)}\right)
\quad\text{as $t \to \infty$}$ if in addition $t/f(1/t) \to \infty$ as $t \to \infty.$
\end{aufzi}
\end{thm}

\begin{proof}
We let $g := 1/f$ and apply Corollary \ref{hpfc.c.app-dom}  to
conclude that $x = g(A)y$. Then  Theorem
\ref{rates.t.rate-indiv} and Proposition \ref{rea.p.comp} imply that
\[ \norm{\Ce_t(A)x}={\rm O}(g(1/t))= {\rm O}\left(\frac{1}{f(1/t)}\right)
\qquad\text{as $t \to \infty$}.
\]
Now suppose that $t/f(1/t) \to \infty$ as $t \to \infty$. In this case, by
Proposition \ref{rea.p.comp}, $t r(t) \to \infty$ as $t \to
\infty$. Furthermore, $x \in Y := \cls{\ran}(A)$ and $Y$ is
$(T(s))_{s\ge 0}$-invariant, so $y \in Y$ as well.
This means that
we can suppose without loss of generality that $(T(s))_{s \ge 0}$
is mean ergodic. Hence the second part of Theorem
\ref{rates.t.rate-indiv} yields that
\[ \norm{\Ce_t(A)x} = \rmo( r(t)) = \rmo(1/f(1/t)) \quad
\text{as $t \to \infty$},
\]
again by Proposition \ref{rea.p.comp}.
\end{proof}
\begin{rem}\label{hirsch}
Let $f$ be a Stieltjes function  with the representation (cf.
\eqref{hpfc.e.stieltjes})
\[
f(z)=\int_{0+}^{\infty} \frac{\mu(\ud{s})}{s+z}, \qquad (z >0),
\quad \text{with}\quad \int_{0+}^{\infty} \frac{\mu(\ud{s})}{s+1} <\infty.
\]
Then
\[
f(z)=\int_{0}^{\infty}e^{-z t}\int_{0+}^{\infty} e^{-t s} \,
\mu(\ud{s}) \, \ud{t} = \int_{0}^{\infty}e^{-z t} m(t) \, \ud{t},
\]
where $m$ is a completely monotone function such that
\[\int_{0}^{1} m(t) \,\ud{t} < \infty \quad \text{and} \quad \lim_{t \to \infty}
m(t)=0.
\]
Hirsch proved in \cite[Corollaire, p. 214-215]{HirFA} that {\em if
$\ran (A)$ is dense in $X$, the following statements are
equivalent for $x\in X$:
\begin{aufzii}
\item  weak  $\displaystyle{\lim_{\alpha \searrow 0}
\int_0^\infty e^{-\alpha s} T(s)x\,
m(s)\, \ud{s}}$ exists;
\item  $\displaystyle{\lim_{\alpha \searrow 0} \int_0^\infty e^{-\alpha s} T(s)x \, m(s)
\, \ud{s}}$ exists;
\item  weak $\displaystyle{\lim_{M \to \infty} \int_0^M  T(s)x \, m(s)\, \ud{s}}$
exists;
\item  $\displaystyle{\lim_{M \to \infty} \int_0^M  T(s)x \,
m(s)\ud{s}}$ exists.
\end{aufzii}
 Moreover, all limits in {\rm (i)-(iv)} are equal to each other.}

\smallskip
We note the following: since Stieltjes functions are potentials,
we can apply Theorem \ref{psr.c.final} to see that (i) implies $x
= g(A)y$, for $g = 1/f$, where $y$ is the limit in (i). On the
other hand, $A$ is injective (since $\ran(A)$ is dense) and $f(z)=
q(z)/z$ for some Bernstein function $q$. Hence the function $f$
belongs to the extended HP-calculus for $A$. General functional
calculus rules then yield that $x \in \dom(f(A))$ and $f(A)x = y$
\cite[Corollary~1.2.4]{Haa2006}. The point is now that one can
pass, conversely,  from $x\in \dom(f(A))$ and $f(A)x = y$ to (iv).
However, Hirsch's proof for this is based essentially on the
functional calculus for sectorial operators, and is beyond the
scope of the present article. The issue will be thoroughly
addressed in a subsequent paper.
\end{rem}

Propositions \ref{rea.p.comp} and  \ref{rea.p.sbf} are
not only useful to determine $r$ from values
of $g$, but also to see that under some weak spectral conditions
on $A$ the rate $r$ is indeed {\em optimal} on $\ran(g(A))$.
The following result illustrates what we mean by this.

\begin{thm}\label{rea.t.main}
Let $g$ be a special Bernstein function (e.g., a complete
Bernstein function) with associated rate function $r=r[g]$
such that
\[
t r(t)\to\infty\;\;\mbox{as}\;\;t\to\infty.
\]
Let $-A$ be the generator of a bounded $C_0-$semigroup $(T(s))_{s\geq 0}$
such that $z=0$ is an accumulation point of $\sigma(A)$. Then, whenever
$\epsilon: (0, \infty) \to (0, \infty)$ is a
decreasing function
with $\lim_{t\to \infty}\epsilon(t)= 0$,
there exists $y\in \ran(g(A))$ such that
\begin{equation}\label{supr}
\sup_{t\geq 1}\frac{\norm{\Ce_t(A)y}}{\epsilon(t)r(t)}=\infty.
\end{equation}
\end{thm}

\begin{proof}
By hypothesis we find $0 \not= z_n=\abs{z_n}e^{i\theta_n}\in \sigma(A), n \in \mathbb N,$ with
$\theta_n\in [-\pi/2,\pi/2]$ and  $z_n\to 0$ as $n\to\infty$. Then
\[
t_n:=1/\abs{z_n}\to\infty\qquad (n\to\infty).
\]
Since $tr(t)\to\infty$, $t\to\infty$,
we may replace
$\epsilon(t)$ by
$\max(\epsilon(t),[tr(t)]^{-1})$, $t\geq 1$
and suppose without loss of generality that
\[
\beta:=\inf_{n\in \N}\,\epsilon(t_n)t_n r(t_n)>0.
\]
Furthermore, $\delta := \inf_{n \in \N} \big|1 - e^{-e^{i\theta_n}}\big| > 0$,
since the function $\theta \mapsto |1 - e^{-e^{i\theta}}|$
is continuous and does not have a zero in $[-\pi/2, \pi/2]$.

By the spectral inclusion Theorem \ref{hpfc.t.spin}, Theorem
\ref{hpfc.t.sbf-est} and Proposition  \ref{rea.p.sbf} we obtain
\begin{align*}
\norm{g(A)\Ce_{t_n}(A)} & = \norm{(\Ce_{t_n} \cdot g)(A)}
\geq \sup_{\lambda\in \sigma(A)}\,
\abs{(\Ce_{t_n}\cdot g)(\lambda)}
\geq
\abs{ \Ce_{t_n}(z_n)g(z_n)}
\\ & =\frac{\big|1-e^{-t_nz_n}\big|}{t_n\abs{z_n}} \abs{g(z_n)}
=  \big|1-e^{-e^{i\theta_n}}\big| \cdot
\abs{g(z_n)} \geq \frac{\delta}{3e^2} r(t_n)
\end{align*}
for each $n\in \N$.
Consequently, since $\epsilon(t_n)\to 0$ as $n\to\infty$, one has
\begin{align*}
\sup_{n\in \N}\,\frac{\norm{g(A)\Ce_{t_n}(A)}}{\epsilon(t_n)r(t_n)}=\infty.
\end{align*}
Each operator $g(A)\Ce_{t_n}(A)$ is similar to its restriction to
$\dom(A)$ by means of the isomorphism $(I+A)^{-1}:X \to \dom(A).$
By the uniform boundedness principle there is $x\in
\dom(A)\subseteq \dom(g(A))$ such that
\[
\sup_{n\in \N}\,\frac{\norm{g(A)\Ce_{t_n}(A)x}_{\dom(A)}}{\epsilon(t_n)r(t_n)}
=\infty.
\]
On the other hand, setting $y:=g(A)x$
we obtain
\begin{align*}
\norm{g(A)\Ce_{t_n}(A)x}_{\dom(A)} & =
\norm{g(A)\Ce_{t_n}(A)x} +\norm{A\Ce_{t_n}(A)g(A)x}
\\ & \le
\norm{\Ce_{t_n}(A)y} +
\frac{M+1}{t_n}\norm{g(A)x}
\\ & \leq
\norm{\Ce_{t_n}(A)y} +
\frac{(M+1)\epsilon(t_n)r(t_n)}{\beta}\norm{g(A)x},
\end{align*}
where as always $M:=\sup_{s\geq 0}\,\norm{T(s)}$.
It follows that
\[
\sup_{n\in\N}\,\frac{\norm{\Ce_{t_n}(A)y}}{\epsilon(t_n)r(t_n)}=\infty,
\]
and this concludes the proof.
\end{proof}

\begin{remark}\label{remarkg}
Clearly, \eqref{supr} can be rewritten as
$$
\sup_{t\geq 1}\frac{\norm{\Ce_t(A)y}}{\epsilon(t)g(1/t)}=\infty.
$$
\end{remark}

\begin{remark}\label{remarkonrates}
For general Bernstein functions $g$ with rate $r= r[g]$ satisfying
$t r(t) \to \infty$ as $t \to \infty$ the conclusion of Theorem
\ref{rea.t.main}
remains true if one requires the stronger spectral condition that
{\em $z=0$ is an accumulation point of
\[ \sigma(A) \cap \{ z \in \C_{+} \suchthat \abs{\arg(z)} \le \theta\}
\]
for some angle $\theta \in [0, \pi/2)$.} (The proof is similar to
that of Theorem \ref{rea.t.main}, but employs Proposition
\ref{rea.p.comp} instead of Proposition  \ref{rea.p.sbf}.) This
applies in particular to semigroups of non-invertible isometries,
since then the whole halfplane $\overline{\C}_+$ is contained in
$\sigma(A)$. A fortiori, it applies also when one has a closed
invariant subspace where the semigroup is like that.
\end{remark}

\section{Examples}\label{s.exas}

We now discuss several examples important for applications. Let in
this section $-A$ be the generator of a bounded $C_0$-semigroup
$(T(s))_{s \ge 0}$ on a Banach space $X.$
For a general theory of fractional powers and
logarithms of $A$ we refer to \cite[Chapter 3]{Haa2006}.

\subsection*{Fractional Powers and Polynomial Rates}

For $0 < \alpha < 1$ the function $g(z) := z^\alpha$ is a
complete Bernstein function with representation
\[ z^\alpha = \int_0^\infty (1 - e^{-zs})\, \mu(\ud{s}),\qquad
\mu(\ud{s}) = \frac{\alpha}{\Gamma(1- \alpha)} s^{-(\alpha + 1)}
\, \ud{s},
\]
where $\Gamma(\cdot)$ is the Gamma function, see \cite[p.
219]{SchilSonVon2010}. The operator $g(A) = (z^\alpha)(A)$ equals
the commonly used fractional power $A^\alpha$ of $A$, see
\cite[Chapter~3]{Haa2006}. The associated potential is
\[ z^{-\alpha} = \frac{1}{\Gamma(\alpha)} \int_0^\infty t^{\alpha -1}
e^{-tz}\, \ud{t}.
\]
By Proposition  \ref{rea.p.comp}, the associated rate satisfies
$r(t) \sim (1/t)^\alpha = t^{-\alpha}$ (take $z_t = 1/t$). Theorems
\ref{rates.t.rate-indiv} and \ref{rea.t.main} then yield the following result.

\begin{thm}\label{exas.t.pol}
Let $-A$ be the generator of a bounded $C_0$-semigroup on a Banach
space $X$ and let $\alpha \in (0,1).$
\begin{aufzi}
\item  For each $x\in \ran(A^\alpha)$
\[ \norm{\Ce_t(A)x} = \rmO(t^{-\alpha}) \qquad
\text{as $t \to \infty$.}
\]
\item  If the semigroup is mean-ergodic, then for each $x\in \ran(A^\alpha)$
\[
\norm{\Ce_t(A)x} = \rmo(t^{-\alpha}) \qquad \text{as $t \to
\infty$.}
\]
\item  If $z=0$ is an accumulation point of $\sigma(A),$ then for any decreasing function
$\epsilon: (0, \infty) \to (0, \infty)$
with $\lim_{t\to \infty}\epsilon(t)= 0$
there exists $y\in \ran(A^\alpha)$ such that
\[
\sup_{t\geq 1}\, \frac{t^{\alpha}  \norm{\Ce_t(A)y}}{\epsilon(t)}=\infty.
\]
\end{aufzi}
\end{thm}

If $\ker(A)=\{0\},$ then one can clearly formulate condition  $x\in \ran(A^\alpha)$ in the
above result as $x \in \dom (A^{-\alpha}).$ Thus, the result can
be given a form similar to Theorem \ref{exas.t.log} on logarithmic
rates below.

\subsection*{Logarithmic Rates}

To find a Bernstein function $g$ with an associated rate $r(t) \sim
1/ \log t$ as $t \to \infty$ we consider the
function
\[ g(z) = \frac{z-1}{\log z}, \quad z >0.
\]
By Example  \ref{hpfc.exa.log}, $f = 1/g$ is a Stieltjes function, whence
by Theorem \ref{hpfc.t.stieltjes-cbf} $g$
 is a complete Bernstein function $g \sim (0, 0, \mu)$.
However, a closed expression for $\mu$ seems to be unknown,
cf.~\cite[p. 236-237]{SchilSonVon2010}. By inserting $z_t = 1/t$ we
obtain $r(t) \sim \frac{1 - 1/t}{\log t}$ for the corresponding
rate function  (Proposition \ref{rea.p.comp}). Hence, by Theorem
\ref{rates.t.rate-indiv} we have
\[ \norm{\Ce_t(A)x} = \rmO(1/\log t)\qquad
\text{as $t \to \infty$ for each $x\in \ran(g(A))$}
\]
and $1/\log t$ is optimal on $\ran(g(A))$ under the spectral conditions
of Theorem  \ref{rea.t.main}.

\subsection*{The Operator Logarithm}

Suppose that $A$ is injective. Then Theorem \ref{exas.t.pol} tells that
if $0 < \alpha < 1$ one has the rate ${\rm O}(t^{-\alpha})$ for
$\Ce_t(A)x$ and   $x \in \dom(A^{-\alpha})$, and this rate is optimal
in the sense of part (iii) of Theorem \ref{exas.t.pol}.
We claim that the analogous result holds with
a logarithmic rate for $x \in \dom(\log A)$.

To begin with, let us say a few words on the operator logarithm.
Since $\log z/(z-1)$ is  Stieltjes (Example \ref{hpfc.exa.log}),
the function $(z \log z)/(z-1)$  is a complete Bernstein function
(Theorem \ref{hpfc.t.stieltjes-cbf}). But then by Lemma
\ref{hpfc.l.bf-fc}
\[\Big(\frac{z \log z}{ z-1}\Big) \frac{1}{1 + z}
\in \Wip(\C_+).
\]
Since
\[  \frac{z-1}{z+1} = 1 - \frac{2}{1+z}\in \Wip(\C_+),\]
we have
\[ \frac{z}{(1 + z)^2} \log z = \Big(\frac{z-1}{z+1}\Big)
\Big(\frac{z \log z}{ z-1}\Big) \frac{1}{1 + z}
\in \Wip(\C_+).
\]
Moreover
\[ \frac{z}{(1+z)^2} = \frac{1}{1 + z} - \frac{1}{(1+z)^2} \,\,\in \Wip(\C_+).
\]
Hence, if $A$ is injective, then  $z(1+z)^{-2}$ is a regularizer for $\log z$,
and therefore $\log A$ is defined in the extended HP-calculus for $A$.

One can approach the operator $\log A$ also via resolvents.
Namely, for any fixed $\lambda \in \C$ with $\abs{\im \lambda} >
\pi$ we have the following two representations of the function
$(\lambda - \log z)^{-1}$ for  $z \in \mathbb C_+$:
\begin{align*}
\frac{1}{\lambda - \log z} & = \int_0^\infty \frac{-1}{(\lambda - \log t)^2 + \pi^2}
\frac{\ud{t}}{t +z}
=\int_0^\infty \Big[ \int_0^\infty
\frac{-e^{-ts} \, \ud{t} }{(\lambda - \log t)^2 + \pi^2} \Big] \, e^{-sz}\, \ud{s}.
\end{align*}
The first is the classical (Stieltjes type) representation used by
Nollau \cite{Nollau1969}
to define $\log A$ for sectorial operators $A$. It is proved by a standard contour
deformation argument, cf.~\cite[Lemma 3.5.1]{Haa2006}. The second (Laplace type)
representation follows easily from the first. It is important for us
since
\[ \int_0^\infty  \int_0^\infty \Big|
\frac{-e^{-ts}}{(\lambda - \log t)^2 + \pi^2} \Big| \,\ud{t}\,
\ud{s} =  \int_0^\infty \frac{1}{ \abs{ (\lambda - \log t)^2 +
\pi^2}} \,\frac{\ud{t}}{t} < \infty.
\]
This shows that $(\lambda - \log z)^{-1} \in \Wip(\C_+)$. From abstract
functional calculus theory \cite[Cor.~1.2.4]{Haa2006} it follows that
\[ \Big( \frac{1}{\lambda - \log z }\Big)(A) = (\lambda - \log A)^{-1},
\]
and that our definition of $\log A$ yields the same operator as
the sectorial functional calculus \cite[Remark~3.3.3]{Haa2006}.

\smallskip

Let us now turn to the problem of rates for $\Ce_t(A)x$ when
$x\in \dom(\log A)$. Since $(z-1)/\log z$ is a Bernstein function, the
following abstract result is useful.

\begin{thm}\label{exas.t.gen}
Let $-A$ be the generator of a bounded
$C_0$-semigroup $(T(s))_{s \ge 0}$ on a Banach space $X$
with $M := \sup_{s \ge 0} \norm{T(s)}$. Furthermore, let
$f$ be a function such that $f(A)$ is defined in the extended HP-calculus
and, for some $0 \not= \lambda\in \C$,
\[ g(z) := \frac{z-\lambda}{f(z)} \quad \text{is a Bernstein function.}
\]
Then the following statements hold:
\begin{aufzi}
\item  For each $x\in \dom(f(A))$
 \[ \norm{\Ce_t(A)x}\le  \frac{c M }{\abs{f(1/t)}}
 \, (\norm{x} + \norm{f(A)x}) \qquad
(t \ge 1),
\]
where the constant $c$ depends only on $f$ and $\lambda$.
\item  If the semigroup
is mean ergodic and $t /\abs{f(1/t)} \to \infty$ as $t \to \infty$,
then
\[
\norm{\Ce_t(A)x} = \rmo\left( \frac{1}{\abs{f(1/t)}}\right)
\quad \text{for each $x\in  \dom(f(A))$.}
\]
\end{aufzi}
\end{thm}

\begin{proof}
We write $r = r[g]$.
From the definition of $g$ we obtain $\lambda = z - g(z)f(z)$. Multiplying with
$\Ce_t(z)$ yields
\[ \lambda \Ce_t (z) = \frac{tz \Ce_t (z)}{t} - (g(z)\Ce_t (z)) f(z) \qquad (t > 0)
\]
and inserting $A$  yields
\[ \abs{\lambda} \norm{\Ce_t(A)x}
\le \frac{M+1}{t} \norm{x} + 2Mr(t) \norm{f(A)x} \qquad (t > 0,\, x\in \dom(f(A))).
\]
Since $r(t) \ge r(1)/t$ for $t \ge 1$ by Theorem \ref{rates.t.char} we arrive at
\[  \norm{\Ce_t(A)x} \le  r(t) \, \frac{2M}{\abs{\lambda}}
\left( \frac{1}{r(1)}  + 1\right) \, \big( \norm{x} + \norm{f(A)x}\big)
\qquad(t \ge 1).
\]
By Proposition  \ref{rea.p.comp} and its proof we have
\[ r(t) \le e \abs{g(1/t)} = e \frac{ (1/t) - \lambda}{\abs{f(1/t)}}
\le \frac{e(1 + \abs{\lambda})}{\abs{f(1/t)}}
\]
Combining this with the previous we obtain
\[ \norm{\Ce_t(A)x} \le \frac{1}{\abs{f(1/t)}} \,
\frac{4Me}{\abs{\lambda}}
\left( \frac{1}{r(1)}  + 1\right) \, \big( \norm{x} + \norm{f(A)x}\big)
\qquad(t \ge 1)
\]
for $x\in \dom(f(A))$, proving a).
Then b) follows   from the above
computations by an argument similar to that from
the proof of Theorem \ref{rates.t.rate-indiv}.
\end{proof}

\begin{remark}
Let  $g\not=0$ be a complete Bernstein function with $\lim_{t \nearrow \infty}
g(t)/t = 0$. Then by Theorem \ref{hpfc.t.cbf-sf} it follows
that $f(z):= g(z) - g(1/z)$ is such that $f(z)/(z-1)$ is Stieltjes.
As in the case of the logarithm, it follows from Theorem \ref{hpfc.t.stieltjes-cbf}
that $(z-1)/f(z)$  and  $q(z) := z f(z)/(z-1)$ are
complete Bernstein functions. Hence
\[ \frac{z}{(1+z)^2}f(z) = \Big(\frac{z-1}{z+1}\Big)\, q(z)\,
 \frac{1}{1 +z} \in \Wip(\C_+),
\]
and as above it follows that if $A$ is injective, then $f(A)$ is defined
in the extended HP-calculus for $A$.
Hence, for injective $A$,  Theorem \ref{hpfc.t.cbf-sf}
provides a large class of functions satisfying the conditions of
Theorem \ref{exas.t.log}.
\end{remark}

We are now in a position to state and prove our final result.

\begin{thm}\label{exas.t.log}
Let $-A$ be the generator of a bounded $C_0$-semigroup on a Banach
space $X$, and suppose that $A$ is injective.
\begin{aufzi}
\item  For each $x\in \dom(\log A)$
\[ \norm{\Ce_t(A)x} = \rmO\Big(\frac{1}{\log t}\Big) \qquad
\text{as $t \to \infty$.}
\]
\item If $\ran(A)$ is dense in $X$ then for each $x\in \dom(\log A)$
\[ \norm{\Ce_t(A)x} = \rmo\Big(\frac{1}{\log t}\Big) \qquad
\text{as $t \to \infty$.}
\]
\item  If $z=0$ is an accumulation point of $\sigma(A),$ then for any decreasing
function $\epsilon: (0,
\infty) \to (0, \infty)$ with $\lim_{t\to \infty}\epsilon(t)=
0$ there exists $y\in \dom(\log A)$ such that
\[
\sup_{t\geq 1}\, \frac{\log t \norm{\Ce_t(A)y}}{\epsilon(t)}=\infty.
\]
\end{aufzi}
\end{thm}

\begin{proof}
As pointed out above, $(z-1)/\log z$ is a Bernstein function,
so the statements a)  and b) follow from Theorem \ref{exas.t.gen}.
The argument for c) follows closely the proof of Theorem \ref{rea.t.main}.
We fix $\tau > \pi$ and $\lambda := i\tau$.
Suppose that $0 \not=z = \abs{z} e^{i\theta} \in \sigma(A)$ such that
$\abs{z} \le e^{-(\tau + \pi/2)}$.
Then $\abs{\theta}\le \pi/2$ and
\[ \abs{\lambda - \log z}^2 = (\log \abs{z})^2 + (\tau - \theta)^2
\le (\log \abs{z})^2 + \Big(\tau + \frac{\pi}{2}\Big)^2
\le 2 (\log\abs{z})^2.
\]
Since  $(\lambda - \log z)^{-1} \in \Wip(\C_+)$ we can apply the spectral inclusion
Theorem  \ref{hpfc.t.spin} and infer --- with $t := 1/\abs{z}$ --- that
\[ \norm{\Ce_t(A)(\lambda- \log A)^{-1}}
\ge \abs{\frac{\Ce_t(z)}{\lambda - \log z}}
\ge \frac{\delta}{\sqrt{2} \abs{\log\abs{z}}},
\]
where $\delta := \inf_{\abs{\varphi} \le \pi/2}
  \big|1 - e^{-e^{i\varphi}}\big| > 0$ as in the proof of
Theorem \ref{rea.t.main}.

Now, by assumption there is a sequence $0 \not= z_n \in \sigma(A)$
with $z_n \to 0$. Without loss of generality we may suppose that
$\abs{z_n} \le e^{-\tau - \pi/2}$ for all $n \in \N$.
Hence, with $t_n := 1/\abs{z_n} \to \infty$,
\[ \norm{\Ce_{t_n}(A) (\lambda - \log A)^{-1}} \ge
\frac{\delta}{\sqrt{2} \log t_n} \qquad (n\in \N).
\]
The uniform boundedness principle yields $x \in X$ such that for
$y := (\lambda - \log A)^{-1}x \in \dom(\log A)$ we have
\[
\sup_{n\in \N}\,\frac{\log (t_n) \norm{\Ce_{t_n}(A)x}}{\epsilon(t_n)}
=\infty,
\]
and this concludes the proof.
\end{proof}

\appendix

\section{}

\begin{thm}\label{int.t.rates1}
Let $-A$ be the generator of a  bounded $C_0$-semigroup
$(T(s))_{s\geq 0}$ on a Banach space $X$.
Suppose $\varphi:(0,\infty) \to
(0,\infty)$ is a function such that $\varphi(t) \searrow 0$ as $t\to\infty$, and
\[
\norm{\Ce_t(A)x} = \rmO(\varphi(t)) \quad \mbox{for every}\;\;x\in
\dom_\infty(A):=\cap_{n=0}^\infty \dom(A^n).
\]
Then
$A$ is invertible.
\end{thm}

\begin{proof}
By assumption it follows that $\dom_{\infty}(A)\subset \cls{\ran}(A).$ Since $-A$ is the generator of a $C_0$-semigroup,
$\dom_{\infty}(A)$ is dense in $X$, so that $\cls {\ran}(A)=X.$ Thus
since $\ker(A)=\{0\},$ it suffices to prove that $\ran (A)$ is closed.
By hypothesis,
for $x\in \dom_\infty(A)$ there is $c(x)$ such that
\begin{equation}\label{app.e.bound}
\norm{\Ce_t(A)x} \leq c(x)\varphi(t) \qquad (t\geq 1).
\end{equation}
For $n \in \N$ we consider $\dom(A^n)$ as a Banach space with the graph norm
\[
\norm{x}_{\dom(A^n)}:= {\sum}_{j=0}^n \norm{A^j x}\qquad (x\in \dom(A^n)).
\]
The space $\dom_\infty(A)$ is a Fr\'echet space with respect to
the increasing sequence of norms $(\norm{\cdot}_{\dom(A^n)})_{n \ge 0}$.
By \eqref{app.e.bound} and the principle of uniform boundedness for
Fr\'echet spaces \cite[Theorem 2.6]{RudFA}, we obtain that there exist
$n\in \N \cup \{0\}$ and $c>0$ such that
\[
\norm{\Ce_t(A)x} \leq c \varphi(t) \norm{x}_{\dom(A^n)}\quad \text{for all}
\quad t\geq 1,\,
x\in \dom_\infty(A).
\]
Since $\dom_\infty(A)$ is a core for the closed operator $A^n$, we
then have
\begin{equation}\label{app.e.bound-n}
\norm{\Ce_t(A)x} \leq  c \varphi(t) \norm{x}_{\dom(A^n)}\quad
\text{for all}\quad t\geq 1,\, x\in \dom(A^n).
\end{equation}
Now  $\Ce_t(A)X \subseteq \dom(A)$ and
\[ tA \Ce_t(A) = A \int_0^t T(s)\, \ud{s} = \Id - T(t) \quad \text{for all}
\quad t > 0.
\]
Hence we have  for every $1\le j\le n$
\[ \norm{A^j\Ce_t x}  \le \frac{\norm{\Id - T(t)}}{t} \norm{A^{j-1}x}
\le \frac{M+1}{t} \norm{A^{j-1}x} \qquad
(x\in \dom(A^n),\, t > 0),
\]
where $M:=\sup_{s \ge 0} \norm{T(s)}$.
From \eqref{app.e.bound-n} it then follows
that
\[ \norm{\Ce_t(A)}_{\dom(A^n) \to \dom(A^n)} \le c\vphi(t) + \frac{M+1}{t},
\]
where the right-hand side tends to zero as $t \to \infty.$
Therefore,  for $t\geq 1$ large enough the operator
\[ \Id - \Ce_t(A):\dom(A^n)\pfeil \dom(A^n)
\]
is invertible. This operator is similar to the operator
\[ \Id - \Ce_t(A): X \pfeil X
\]
by virtue of the isomorphism $(\Id + A)^{-n}: X \to \dom(A^n)$.
Multiplying with $t$ yields that $S := \int_0^t (\Id - T(s))\, \ud{s}$
is invertible on $X$. But for $x \in \dom(A)$
\[ Sx = \int_0^t (\Id - T(s))x\, \ud{s} = \int_0^t \int_0^s T(r)Ax\, \ud{r}\, \ud{s}
= \int_0^t (t - r)T(r)Ax\, \ud{r},
\]
and hence
\[ \norm{x} \le \norm{S^{-1}} \Big\|\int_0^t (\Id - T(s))x \, \ud{s} \Big\|
\le (CMt^2/2) \norm{Ax}
\]
for large enough $t$ and for all $x\in \dom(A)$. This concludes
the proof.
\end{proof}

\section{} \label{AppendixB}

In this Appendix we characterize those  rate functions $r$
that appear as $r=[g]$ for some
{\em complete} Bernstein function. We need
the following lemma.

\begin{lemma}\label{l.rate-invlap}
Let $g$ be a Bernstein function and
let $r = r[g]$ the associated rate as in \eqref{rates.e.r}.
Then, with $\tilde{r}(t) = tr(t)$,
\[ (\Lap \tilde{r})(z) = \frac{g(z) - \frac{g(0+)}{2}}{z^2} \qquad (\re z > 0).
\]
In other words, $tr(t)$ is the inverse Laplace transform of $[g(z)- g(0+)/2]/z^2$.
\end{lemma}

\begin{proof}
Suppose first that $g(0+)= 0$, so that
$g(\lambda)=b\lambda +\int_{0+}^\infty (1-e^{-\lambda s})\mu(\ud{s})$.
Then
\begin{align*}
\int_0^\infty & tr(t)e^{-\lambda t}\,\ud{t}
= \int_0^\infty
e^{-\lambda t} \left[b+ \int_{0+}^t s \mu(\ud{s})\,\ud{t}
+ t\int_{t+}^\infty
\mu(\ud{s})\right]\,\ud{t}\\
&=\frac{b}{\lambda}+\int_{0+}^\infty \left(\int_s^\infty
e^{-\lambda t}\,\ud{t}\right)\,s \mu(\ud{s})+ \int_{0+}^\infty
\left(\int_0^s t e^{-\lambda t} \ud{t}\right)\, \mu(\ud{s})\\
&=\frac{b}{\lambda}+\frac{1}{\lambda}\int_{0+}^\infty s
e^{-\lambda s}\,\mu(\ud{s})+ \int_{0+}^\infty \left(-\frac{s
e^{-\lambda s}}{\lambda}+ \frac{1-e^{-\lambda
s}}{\lambda^2}\right)\, \mu(\ud{s})
\\
&=\frac{b}{\lambda}+\frac{1}{\lambda^2} \int_{0+}^\infty
(1-e^{-\lambda s})\, \mu(\ud{s})=\frac{g(\lambda)}{\lambda^2}
\end{align*}
for all $\lambda>0$. In the general case we have to add
$\frac{a}{2} \int_0^\infty t e^{-t\lambda}\, \ud{t} = \frac{g(0+)}{2\lambda^2}$
in each step of the computation.
\end{proof}

\begin{cor}
If $r= r[g]$ for a Bernstein function $g$ such that $$g(0+) = \lim_{t \to \infty}r(t) = 0,$$  then $t r(t)$ is
the inverse Laplace transform  of $g(z)/z^2$.
\end{cor}

Now we can state the main result of this Appendix.

\begin{thm}\label{t.cbf-rate}
A function $r : (0, \infty) \to (0, \infty)$ is of the form
$r = r[g]$ for some complete Bernstein function $g$ if, and only if,
$t \mapsto tr(t)$ is a Bernstein function.
\end{thm}

\begin{proof}
Suppose that $r = r[g]$ for some complete
Bernstein function $g$. Then  $g - g(0+)/2$ is
a complete Bernstein function, too, whence
by \cite[Theorem 6.2]{SchilSonVon2010} there exists
a Bernstein function $h$ such that
\[ g(\lambda) - \frac{g(0+)}{2} = \lambda^2
(\Lap h)(\lambda)\qquad (\lambda > 0).
\]
Lemma \ref{l.rate-invlap}
together with the uniqueness  of the Laplace transform
implies that $h(t) = tr(t)$ for $t > 0$, and hence $t \mapsto tr(t)$
is a Bernstein function.

Conversely, suppose that $h(t) = tr(t)$
is a Bernstein function.
Then $a := 2 \lim_{t\to \infty} h(t)/t$ exists
and   $h$ is Laplace transformable. By
\cite[Theorem 6.2]{SchilSonVon2010} again,
\[ g(\lambda) := \frac{a}{2} + \lambda^2 (\Lap h)(\lambda)
\qquad (\lambda > 0)
\]
is a complete Bernstein function. A short computation based
on the Bernstein representation of $h$ yields that
\[ \lim_{\lambda \searrow 0} \lambda^2 (\Lap h)(\lambda)
= \lim_{t\to \infty} \frac{h(t)}{t} = \frac{a}{2},
\]
cf.~the proof of \cite[Theorem 6.2]{SchilSonVon2010}. It follows
that $a = g(0+)$ and, by Lemma \ref{l.rate-invlap} and the uniqueness of the
Laplace transform, $r[g](t) =h(t)/t = r(t)$ for all $t > 0$.
\end{proof}

\vspace{0,2cm}

\begin{center}
{\bf Acknowledgements}
\end{center}
\vspace{0,2cm}

The authors are grateful to the referee for his/her very careful reading of the manuscript and
useful remarks and suggestions.
\vspace{0,2cm}


\begin{thebibliography}{100}
\bibitem{AsLi07}
 I. Assani and M. Lin,
\emph{On the one-sided ergodic Hilbert transform}, Ergodic theory
and related fields, Contemporary  Math., \textbf{430}, Providence, RI,
2007, 21--39.


\bibitem{Bala59}
A. V.~Balakrishnan, \emph{An operational calculus for
infinitesimal generators of semigroups,} Trans. Amer. Math. Soc.
\textbf{91} (1959), 330--353.


\bibitem{Bro58} F. E.~Browder,
\emph{On the iteration of transformations in noncompact minimal
dynamical systems}, Proc. Amer. Math. Soc. {\bf 9} (1958),
773--780.


\bibitem{BuGe95} P.~Butzer and A. Gessinger, \emph{
Ergodic theorems for semigroups and cosine operator functions at
zero and infinity with rates; applications to partial differential
equations. A survey,} Mathematical analysis, wavelets, and signal
processing, Contemporary Math. \textbf{190} (1995),  67--94.

\bibitem{BuWe71}
P.~Butzer and U. Westphal, \emph{The mean  ergodic theorem and
saturation}, Indiana Univ. Math. J. \textbf{20} (1971),
1163--1174.

\bibitem{ClPr90}
P. Cl\'ement and J. Pr\"uss,
\emph{Completely positive measures and Feller semigroups},
Math. Ann. \textbf{287} (1990), 73--105.

\bibitem{CoCuLi10}
 G. Cohen, C. Cuny and M. Lin, \emph{The one-sided ergodic Hilbert
transform in Banach spaces}, Studia Mathematica {\bf 196} (2010),
251--263.

\bibitem{CoCuLi11} G. Cohen, C. Cuny and M. Lin, \emph{On convergence of power series of $L_p$
contractions,} Banach Center Publications, 2011, to appear.

\bibitem{CoLi03}
G. Cohen and M. Lin, \emph{Laws of large numbers with rates and
the one-sided ergodic Hilbert transform}, Illinois J. Math.
\textbf{47} (2003), 997--1031.

\bibitem{CoLi05}
G. Cohen and M. Lin, \emph{Extensions of the Menchoff-Rademacher
theorem with applications to ergodic theory}, Israel J. Math. \textbf{148} (2005), 41--86.

\bibitem{CoLi09}
G. Cohen and M. Lin, \emph{The one-sided ergodic Hilbert transform
of normal contractions,} in: Characteristic functions, scattering
functions and transfer functions - the Moshe Livsic memorial
volume. Birkh\"auser, Basel, 2009, 77--98.

\bibitem{Cu09}
C. Cuny, \emph{Pointwise ergodic theorems with rate and
application to limit theorems for stationary processes},
Stochastics and Dynamics, \textbf{11} (2011), 135-155.

\bibitem{Cu10} C. Cuny,
\emph{On the a.s. convergence of the one-sided ergodic Hilbert
transform}, Ergodic Th. Dynamical Sys. {\bf 29}  (2009),
1781--1788.

\bibitem{CuLi09} C. Cuny and M. Lin,
\emph{Pointwise ergodic theorems with rate and application to the
CLT for Markov chains}, Ann. Inst. H. Poincare, ser. Prob. and
Stat. {\bf 45} (3) (2009), 710--733.

\bibitem{deLau95}
R. deLaubenfels, \emph{Automatic extensions of functional calculi},
   Studia Math. \textbf{114} (1995), 237--259.

\bibitem{Der06}
Y. Derriennic, \emph{Some aspects of recent works on limit
theorems in ergodic theory with special emphasis on the ``central
limit theorem''}, Discrete Contin. Dyn. Syst. (Ser. A) {\bf 15}
(2006), 143--158.


\bibitem{DerLin01}
Y. Derriennic and M. Lin, \emph{Fractional Poisson equations and
ergodic theorems for fractional coboundaries}, Israel J. Math.
\textbf{123} (2001), 93--130.


\bibitem{Dun43}
N. Dunford, \emph{Spectral theory. {I}. {C}onvergence to
projections}, Trans. Amer. Math. Soc. {\bf 54} (1943), 185--217.

\bibitem{EN}
K.-J. Engel and R. Nagel, \emph{One-Parameter Semigroups for
Linear Evolution Equations}, Graduate Texts in Mathematics
\textbf{194},  Berlin, Springer, 2000.

\bibitem{Ga98} V. F. Gaposhkin, \emph{
On the rate of decrease of the probabilities of $\epsilon$-deviations for means of stationary processes,}
 Mat. Zametki \textbf{64} (1998),  366--372 (in Russian); translation in
Math. Notes \textbf{64} (1998), 316–-321.


\bibitem{GoRaSho78}
J. Goldstein, C. Radin, and R. Showalter, \emph{Convergence rates
of ergodic limits for semigroups and cosine functions}, Semigroup
Forum \textbf{16} (1978), 89--95.


\bibitem{GoHaTo11}
A. Gomilko, M. Haase, and Yu. Tomilov, \emph{On rates in mean
ergodic theorems}, Math. Res. Letters \textbf{18} (2011), 201--213.




\bibitem{Haa2006} M. Haase, The Functional Calculus for Sectorial Operators.
Operator Theory: Advances and Applications \textbf{169},
Birkh\"auser,  Basel, 2006.



\bibitem{HaTo10} M. Haase and Yu. Tomilov,
\emph{Domain characterizations of certain functions of
power-bounded
 operators}, Studia Math. {\bf 196} (2010), 265--288.

\bibitem{HilPhi}
E. Hille abd R. S. Phillips, \emph{Functional Analysis and
Semi-Groups,} 3rd printing of rev. ed. of 1957, Colloq. Publ.
\textbf{31}, AMS, Providence, RI, 1974.

\bibitem{Hir74}
F. Hirsch, \emph{Transformation de {S}tieltjes et fonctions
op\'erant sur les potentiels abstraits}, Th\'eorie du potentiel et
analyse harmonique ({J}ourn\'ees {S}oc. {M}ath. {F}rance, {I}nst.
{R}echerche {M}ath.  {A}vanc\'ee, {S}trasbourg, 1973),  Lecture
Notes in Math.  \textbf{404}, Springer, Berlin, 1974, 149--163.

\bibitem{Hir75}
F. Hirsch, \emph{ Familles d'op\'erateurs potentiels,} Ann. Inst.
Fourier \textbf{254} (1975), 263--288.

\bibitem{HirFA}
F. Hirsch, {\em Domaines d'op\'erateurs repr\'esent\'es comme int\'egrales de r\'esolvantes},
J. Functional Analysis \textbf{23} (1976), 199–-217.


\bibitem{Ka96}
A. G. Kachurovskii, \emph{Rates of convergence in ergodic
theorems,} Uspekhi Mat. Nauk \textbf{51} (1996), 73--124 (in
Russian); translation in Russian Math. Surveys \textbf{51} (1996),
653–-703.

\bibitem{KaRe10} A. G. Kachurovskii and A. V. Reshetenko, \emph{On the rate of
convergence in von Neumann's ergodic theorem with continuous
time,}  Mat. Sb. \textbf{201} (2010),  25--32 (in Russian);
translation in Sb. Math. \textbf{201} (2010),  493-–500.

\bibitem{Ka28} Th. Kaluza, \emph{ \" Uber die Koeffizienten reziproker
Potenzreihen,} Math. Z. \textbf{28} (1928), 161--170.


\bibitem{Kr85} U. Krengel, Ergodic Theorems. With a
supplement by Antoine Brunel. De Gruyter Studies in Mathematics
\textbf{6},  Berlin, Walter de Gruyter, 1985.

\bibitem{KrLi84} U. Krengel and M. Lin, \emph{On the range of the generator of a Markovian semigroup,}
Math. Z. \textbf{185} (1984), 553--565.




\bibitem{Lin74a} M. Lin,
\emph{On the uniform ergodic theorem}, Proc. Amer. Math. Soc. {\bf
43} (1974), 337--340.


\bibitem{Lin74b} M. Lin,
\emph{On the uniform ergodic theorem. {II}}, Proc. Amer. Math.
Soc. {\bf 46} (1974), 217--225.


\bibitem{Nollau1969}
V. Nollau, \emph{\"{U}ber den {L}ogarithmus abgeschlossener
{O}peratoren in {B}anachschen {R}\"aumen}, Acta Sci. Math.
(Szeged)  \textbf{30} (1969), 161--174.

\bibitem{Phil52}
R. S. Phillips, \emph{On the generation of semigroups of linear
operators,} Pacific J. Math. \textbf{2} (1952), 343–-369.

\bibitem{RudFA}
W. Rudin, \emph{Functional Analysis,} 2nd ed., International
Series in Pure and Applied Mathematics, NY,  McGraw-Hill, 1991.

\bibitem{Shaw} S.-Y. Shaw, Convergence rates of ergodic limits and associated solutions, J. Approximation Theory
\textbf{75} (1993), 157--166.


\bibitem{Schil98}
R. L. Schilling, \emph{Subordination in the sense of {B}ochner and
a related  functional calculus}, J. Austral. Math. Soc. Ser. A
\textbf{64} (1998), 368--396.


\bibitem{SchilSonVon2010}
 R. Schilling, R. Song, and Z. Vondra{\v{c}}ek,  \emph{Bernstein functions},
 de Gruyter Studies in Mathematics \textbf{37}, Walter de Gruyter, Berlin, 2010.


 \bibitem{West1998}
 U. Westphal, \emph{A generalized version of the Abelian mean ergodic theorem with rates for semigroup operators and fractional powers of infinitesimal generators,}
Results in Math. \textbf{34} (1998), 381--394.

\bibitem{Wi41} D. V. Widder, \emph{The Laplace Transform}, Princeton Mathematical Series \textbf{6}, Princeton University Press, Princeton, 1941.
 \end{thebibliography}
  \end{document}